\documentclass[final,notitlepage,12pt,reqno]{amsart}
\usepackage{graphicx,extarrows}
\usepackage{calc}
\usepackage{url}
\newlength{\depthofsumsign}
\makeatletter
\let\I\@undefined
\makeatother

\usepackage{geometry}
\usepackage{indentfirst}
\usepackage[normalem]{ulem}
\usepackage{float}
\usepackage{amsthm}

\usepackage{enumitem}[2011/09/28]
\setenumerate{align=left, leftmargin=0pt,labelsep=.5em, labelindent=0\parindent,listparindent=\parindent,itemindent=*}
\usepackage{array}
\usepackage[all]{xy}

\usepackage{multirow}

\usepackage{caption}[2012/02/19]

\usepackage{longtable,lscape}

\usepackage{colonequals}

\usepackage{amssymb,bm,amsmath}
\usepackage{amsfonts}

\usepackage[french,german,english]{babel}
\usepackage{appendix}

\newcommand{\qoppa}{\kern-.1em\rotatebox{180}{\raisebox{-0.42em}{$'$}}\kern-.2em{o}}

\DeclareMathOperator{\Span}{span}

\DeclareMathOperator{\Li}{Li}

\DeclareMathOperator{\D}{d}
\DeclareMathOperator{\I}{Im}
\DeclareMathOperator{\RE}{Re}

\def\eor{\hfill$ \square$}

\newcolumntype{L}{>{$}l<{$}}
\newcolumntype{C}{>{$}c<{$}}
\newcolumntype{R}{>{$}r<{$}}

\theoremstyle{plain}
\newtheorem{theorem}{Theorem}[section]
\newtheorem{proposition}[theorem]{Proposition}
\newtheorem{lemma}[theorem]{Lemma}
\newtheorem{corollary}[theorem]{Corollary}

\newenvironment{remark}[1][Remark]{\begin{trivlist}
\item[\hskip \labelsep {\bfseries #1}]}{\end{trivlist}}

\theoremstyle{definition}

\numberwithin{equation}{section}

\usepackage{Baskervaldx}
\usepackage[baskervaldx]{newtxmath}
\usepackage[cal=cm,scr=rsfs,frak=euler]{mathalfa}

\usepackage[counter,user]{zref}
\usepackage{refcount}
\makeatletter

\zref@newprop{AAAX}{\the\value{\zref@getcurrent{counter}}}%

\newcommand\AAAXlabel[1]{%
  \zref@labelbyprops{#1}{AAAX}%
  \label{#1}%
}

\zref@newprop{BBBX}{\the\value{\zref@getcurrent{counter}}}%

\newcommand\BBBXlabel[1]{%
  \zref@labelbyprops{#1}{BBBX}%
  \label{#1}%
}

\zref@newprop{AAAY}{\the\value{\zref@getcurrent{counter}}}%

\newcommand\AAAYlabel[1]{%
  \zref@labelbyprops{#1}{AAAY}%
  \label{#1}%
}

\zref@newprop{BBBY}{\the\value{\zref@getcurrent{counter}}}%

\newcommand\BBBYlabel[1]{%
  \zref@labelbyprops{#1}{BBBY}%
  \label{#1}%
}

\zref@newprop{AAAA}{\the\value{\zref@getcurrent{counter}}}%

\newcommand\AAAAlabel[1]{%
  \zref@labelbyprops{#1}{AAAA}%
  \label{#1}%
}

\zref@newprop{BBBB}{\the\value{\zref@getcurrent{counter}}}%

\newcommand\BBBBlabel[1]{%
  \zref@labelbyprops{#1}{BBBB}%
  \label{#1}%
}

\zref@newprop{AAA}{\the\value{\zref@getcurrent{counter}}}%

\newcommand\AAAlabel[1]{%
  \zref@labelbyprops{#1}{AAA}%
  \label{#1}%
}
\newcommand\AAAref[1]{%
   \zref[AAA]{#1}%
}
\zref@newprop{BBB}{\the\value{\zref@getcurrent{counter}}}%

\newcommand\BBBlabel[1]{%
  \zref@labelbyprops{#1}{BBB}%
  \label{#1}%
}
\newcommand\BBBref[1]{%
   \zref[BBB]{#1}%
}
\zref@newprop{AA}{\the\value{\zref@getcurrent{counter}}}%

\newcommand\AAlabel[1]{%
  \zref@labelbyprops{#1}{AA}%
  \label{#1}%
}
\newcommand\AAref[1]{%
   \zref[AA]{#1}%
}
\zref@newprop{BB}{\the\value{\zref@getcurrent{counter}}}%

\newcommand\BBlabel[1]{%
  \zref@labelbyprops{#1}{BB}%
  \label{#1}%
}
\newcommand\BBref[1]{%
   \zref[BB]{#1}%
}
\zref@newprop{CC}{\the\value{\zref@getcurrent{counter}}}%

\newcommand\CClabel[1]{%
  \zref@labelbyprops{#1}{CC}%
  \label{#1}%
}
\newcommand\CCref[1]{%
   \zref[CC]{#1}%
}
\zref@newprop{DD}{\the\value{\zref@getcurrent{counter}}}%

\newcommand\DDlabel[1]{%
  \zref@labelbyprops{#1}{DD}%
  \label{#1}%
}
\newcommand\DDref[1]{%
   \zref[DD]{#1}%
}
\zref@newprop{SS}{\the\value{\zref@getcurrent{counter}}}%

\newcommand\SSlabel[1]{%
  \zref@labelbyprops{#1}{SS}%
  \label{#1}%
}

\zref@newprop{SSO}{\the\value{\zref@getcurrent{counter}}}%

\newcommand\SSOlabel[1]{%
  \zref@labelbyprops{#1}{SSO}%
  \label{#1}%
}

\makeatother
\usepackage[scr=rsfs]{mathalfa}
\DeclareMathAlphabet{\mathsf}{OT1}{\sfdefault}{m}{n}
\SetMathAlphabet{\mathsf}{bold}{OT1}{\sfdefault}{m}{n}
\DeclareSymbolFontAlphabet{\mathbb}{AMSb}

\usepackage{color}

\def\Z{\Bbb Z}

\def\l{\left}
\def\r{\right}
\def\bg{\bigg}
\def\({\bg(}
\def\){\bg)}
\def\t{\text}
\def\f{\frac}

\def\bi{\binom}

\begin{document}

\pagenumbering{roman}
\selectlanguage{english}
\title[Proof of conjectures on series with summands involving $ \binom{2k}{k}8^k/(\binom{3k}{k}\binom{6k}{3k})$]{Proof of conjectures on series\\ with summands involving $ \binom{2k}{k}8^k/(\binom{3k}{k}\binom{6k}{3k})$}
\author{Zhi-Wei Sun}\address{(Zhi-Wei Sun) Department of Mathematics, Nanjing
University, Nanjing 210093, People's Republic of China}
\email{{\tt zwsun@nju.edu.cn}
\newline\indent
{\it Homepage}: {\tt http://maths.nju.edu.cn/\lower0.5ex\hbox{\~{}}zwsun}}
 \author{Yajun Zhou
}
\address{(Yajun Zhou) Program in Applied and Computational Mathematics (PACM), Princeton University, Princeton, NJ 08544} \email{yajunz@math.princeton.edu}\curraddr{\textrm{} \textsc{Academy of Advanced Interdisciplinary Studies (AAIS), Peking University, Beijing 100871, P. R. China}}\email{yajun.zhou.1982@pku.edu.cn}
\date{\today}\thanks{\indent *The first author was supported by the Natural Science Foundation of China (grant no. 12371004), and the second author was supported in part  by the Applied Mathematics Program within the Department of Energy
(DOE) Office of Advanced Scientific Computing Research (ASCR) as part of the Collaboratory on
Mathematics for Mesoscopic Modeling of Materials (CM4)}

\maketitle

\begin{abstract}
     Using cyclotomic multiple zeta values of level $8$, we confirm and generalize  several conjectural  identities on infinite series with summands involving $\bi{2k}k8^k/(\bi{3k}k\bi{6k}{3k})$.
     For example, we prove that
     $$\sum_{k=0}^\infty\frac{(350k-17)\binom{2k}k8^k}
{\binom{3k}k\binom{6k}{3k}}=15\sqrt2\,\pi+27$$
and
$$\sum_{k=1}^\infty\f{\l\{(5k-1)\left[16\mathsf H_{2k-1}^{(2)}-3\mathsf H_{k-1}^{(2)}\right]-\frac{12(6k-1)}{(2k-1)^2}\r\}\bi{2k}k8^k}
{k(2k-1)\bi{3k}k\bi{6k}{3k}}=\f{\pi^3}{12\sqrt2},$$
where $\mathsf H^{(2)}_m$ denotes the second-order harmonic number $\sum_{0<j\le m}\f1{j^2}$.
      \end{abstract}

\pagenumbering{arabic}

\section{Introduction}

For $r\in\mathbb Z_{>0}=\{1,2,\ldots\}$, the $r$th harmonic numbers are given by
\begin{align} \mathsf H_m^{(r)}\colonequals\sum_{k=1}^m\frac1{k^{r}}
 \ \ (m\in\mathbb Z_{\ge0}=\{0,1,2,\ldots\}).\end{align} Customarily, one writes $\mathsf H_m^{}$ as a short-hand for $  \mathsf H_m^{(1)}$. 
 Recently, the first author \cite{Sun2022,Sun2023} has proposed many conjectures on values of infinite series whose summands involve both binomial coefficients and harmonic numbers.
 
 The Dirichlet $L$-function associated with a Dirichlet character $\chi$ is given by
 $$L(s,\chi):=\sum_{n=1}^\infty\f{\chi(n)}{n^s}\ \ \ \t{for}\ \Re(s)>1.$$
 For convenience, we define the constant     
 \begin{align} L\colonequals L\l(2,\l(\f{-8}{\cdot}\r)\r)=\sum_{n=1}^\infty\left(\f{-8}n\right)\f{1}{n^2}
=\sum_{k=0}^\infty\f{(-1)^{k(k-1)/2}}{(2k+1)^2},\end{align}
where $(\f{-8}{\cdot})$ is the Kronecker symbol.

Sun \cite{Sun2013} proved that
\begin{align}2(2k+1)\bi{2k}k\ \bigg|\ \bi{3k}k\bi{6k}{3k}\quad\t{for all}\ k\in\Z_{>0}.\end{align}
In a recent preprint  \cite[\S5]{Sun2023} by the first author,  one can find several conjectures about infinite series with summands involving $\bi{2k}k/(\bi{3k}k\bi{6k}{3k})$, such as {\allowdisplaybreaks\small\begin{align}
\sum_{k=0}^\infty\f{(350k-17)\bi{2k}k8^k}
{\bi{3k}k\bi{6k}{3k}}={}&15\sqrt2\pi+27,\label{eq:conj5.4(1)}\\\sum_{k=1}^\infty\f{\l
[21(350k-17)(2\mathsf H_{6k-1}-\mathsf H_{3k-1}-\mathsf H_{k-1})+4850\r]\bi{2k}k8^k}{\bi{3k}k\bi{6k}{3k}}={}&976+1020\sqrt2\pi+945\sqrt2\pi\log2,\label{eq:conj5.4(2)}\\\sum_{k=1}^\infty\f{\l[7(350k-17)(\mathsf H_{2k-1}-\mathsf H_{k-1})+2225\r]\bi{2k}k8^k
}{\bi{3k}k\bi{6k}{3k}}={}& 276+\f{493\pi}{\sqrt2}+\f{315\pi\log2}{\sqrt2}-420L,\label{eq:conj5.4(3)}\\\sum_{k=1}^\infty\f{[(50k-7)(\mathsf H_{2k-1}-\mathsf H_{k-1})+5]\bi{2k}k8^k}{k\bi{3k}k\bi{6k}{3k}}={}&3\sqrt2\pi(1+\log2)-8L,\label{eq:conj5.5(1)}\\\sum_{k=1}^\infty\f{[(50k-7)(2\mathsf H_{6k-1}-\mathsf H_{3k-1}-\mathsf H_{k-1})-10]\bi{2k}k8^k}{k\bi{3k}k\bi{6k}{3k}}={}&2\sqrt2\pi(2+3\log2),\label{eq:conj5.5(2)}\\\sum_{k=1}^\infty\f{\left[(5k-1)(\mathsf H_{2k-1}-\mathsf H_{k-1})-\frac{3(6k-1)}{4(2k-1)}\right]\bi{2k}k8^k}{k(2k-1)\bi{3k}k\bi{6k}{3k}}={}&\frac{3\sqrt2\,\pi\log2}{8}-L,\label{eq:conj5.5ii1}\\\sum_{k=1}^\infty\f{\left[(5k-1)(2\mathsf H_{6k-1}-\mathsf H_{3k-1}-\mathsf H_{k-1})-\frac{2(3k-1)}{2k-1}\right]\bi{2k}k8^k}{k(2k-1)\bi{3k}k\bi{6k}{3k}}={}&\frac{3\sqrt2\,\pi\log2}{4},\label{eq:conj5.5ii2}\\\sum_{k=1}^\infty\f{\big[(5k-1)(12\mathsf H_{6k-1}-6\mathsf H_{3k-1}-4\mathsf H_{2k-1}-2\mathsf H_{k-1})-9\big]\bi{2k}k8^k}{k(2k-1)\bi{3k}k\bi{6k}{3k}}={}&3\sqrt2\,\pi\log2+4L,\label{eq:conj5.5ii3}
\\\sum_{k=1}^\infty\f{\l\{(5k-1)\left[16\mathsf H_{2k-1}^{(2)}-3\mathsf H_{k-1}^{(2)}\right]-\frac{12(6k-1)}{(2k-1)^2}\r\}\bi{2k}k8^k}
{k(2k-1)\bi{3k}k\bi{6k}{3k}}={}&\f{\pi^3}{12\sqrt2}.\label{eq:conj5.5iii}\end{align}}Here, the summation formula \eqref{eq:conj5.4(1)} immediately entails two more identities
(cf.\ \cite[Remark 5.5]{Sun2023} and the last Remark in \S\ref{subsec:Li_n_char}  below): \begin{align}
\sum_{k=1}^\infty\f{(50k-7)\bi{2k}k8^k}{k\bi{3k}k\bi{6k}{3k}}={}&2\sqrt2\,\pi+4,\label{eq:Remark5.5a}\\\sum_{k=1}^\infty\f{(5k-1)\bi{2k}k8^k}{k(2k-1)\bi{3k}k\bi{6k}{3k}}={}&\f{\pi}{2\sqrt2}.
\label{eq:Remark5.5b}\end{align}
 
Our main purpose is to confirm all these summation formulae and put them into a broader context.
In \S\ref{sec:toolkit}, we set up an analytic framework that is essential to all the subsequent computations, reducing the summations of infinite series to the evaluations of definite integrals. In  \S\ref{sec:polylog6k}, we integrate over certain products of rational functions and logarithmic expressions, not only proving \eqref{eq:conj5.4(1)}--\eqref{eq:conj5.5iii}, but also evaluating a wider class of infinite series in the mean time.

\section{Analytic toolkit\label{sec:toolkit}}
\subsection{Some integrals related to binomial coefficients\label{subsec:6k_int_repn}}In this work, we convert infinite series into definite integrals, using the next lemma. \begin{lemma}
We have the following identities for $ k\in\mathbb Z_{>0}$:{\allowdisplaybreaks\small\begin{align}
\frac{ \binom{2 k}{k}}{4k\binom{3k}k \binom{6 k}{3 k}}={}&\int_0^1   \frac{\left[ \frac{t\left( 1-t^{2} \right)}{2^{2}} \right]^{2k}\D t}{1-t^{2}},\label{eq:Itab2a}\\\frac{ (-\mathsf H_k+2\mathsf  H_{2 k}+\mathsf H_{3 k}-2\mathsf  H_{6 k})\bi{2k}k}{8k\binom{3k}k \binom{6 k}{3 k}}={}&\int_0^1   \frac{\left[ \frac{t\left( 1-t^{2} \right)}{2^{2}} \right]^{2k}\log t\D t}{1-t^{2}},\label{eq:Itab2b}\\\frac{ (\mathsf H_{2 k-1}+\mathsf H_{3 k}-2\mathsf H_{6 k}+2 \log2)\bi{2k}k}{4k\binom{3k}k \binom{6 k}{3 k}}={}&\int_0^1   \frac{\left[ \frac{t\left( 1-t^{2} \right)}{2^{2}} \right]^{2k}\log(1- t^{2})\D t}{1-t^{2}},\label{eq:Itab2c}\\\frac{\left[ \left(\mathsf H_{k-1}-\mathsf H_{2k-1}+2\log2\right)^2 +\mathsf H_{k-1}^{(2)}-5\mathsf H_{2k-1}^{(2)}+\frac{2\pi^2}{3}\right]\bi{2k}k}{4k\binom{3k}{k}\binom{6k}{3k}}={}&\int_{0}^1\frac{\left[ \frac{t\left( 1-t^{2} \right)}{2^{2}} \right]^{2k}\log^2\frac{t^{2}}{1-t^{2}}\D t}{1-t^2},\label{eq:log_sqr_1-tt}\\\frac{\binom{2k}{k}}{(2k-1)\binom{3k}{k}\binom{6k}{3k}}={}&\int_0^1\frac{\left[\frac{t\left(1-t^2\right)}{2^{2}}\right]^{2k}\D t}{t^{2}},\label{eq:2k-1_int_repn1}\\\frac{\left( -\mathsf H_k+2\mathsf  H_{2 k}+\mathsf H_{3 k}-2\mathsf  H_{6 k}-\frac{2}{2 k-1} \right)\bi{2k}k}{2(2k-1)\binom{3k}{k}\binom{6k}{3k}}={}&\int_0^1\frac{\left[\frac{t\left(1-t^2\right)}{2^{2}}\right]^{2k}\log t\D t}{t^{2}},\label{eq:2k-1_int_repn2}\\\frac{\left( \mathsf H_{2 k}+\mathsf H_{3 k}-2\mathsf  H_{6 k}+2 \log 2 \right)\bi{2k}k}{(2k-1)\binom{3k}{k}\binom{6k}{3k}}={}&\int_0^1\frac{\left[\frac{t\left(1-t^2\right)}{2^{2}}\right]^{2k}\log\left(1-t^2\right)\D t}{t^{2}},\label{eq:2k-1_int_repn3}\\\frac{\left[ \left(\mathsf H_{k}-\mathsf H_{2k}+\frac{2}{2k-1}+2\log2\right)^2 +\mathsf H_{k}^{(2)}-5\mathsf H_{2k}^{(2)}+\frac{4}{(2k-1)^{2}}+\frac{2\pi^2}{3}\right]\bi{2k}k}{(2k-1)\binom{3k}{k}\binom{6k}{3k}}={}&\int_{0}^1\frac{\left[ \frac{t\left( 1-t^{2} \right)}{2^{2}} \right]^{2k}\log^2\frac{t^{2}}{1-t^{2}}\D t}{t^2},\label{eq:log_sqr_tt}
\intertext{and the following identities for $ k\in\mathbb Z_{\geq0}$:}\frac{ \binom{2 k}{k}}{(6k+1)\binom{3k}k \binom{6 k}{3 k}}={}&\int_0^1\left[\frac{t(1-t^2)}{2^{2}}\right]^{2 k}\D t,\label{eq:Itab2a'}\\\frac{ (-\mathsf H_k+2\mathsf  H_{2 k}+\mathsf H_{3 k}-2\mathsf  H_{6 k+1})\bi{2k}k}{2(6k+1)\binom{3k}k \binom{6 k}{3 k}}={}&\int_0^1\left[\frac{t(1-t^2)}{2^{2}}\right]^{2 k}\log t\D t,\label{eq:Itab2b'}\\\frac{ (\mathsf H_{2 k}+\mathsf H_{3 k}-2\mathsf H_{6 k+1}+2 \log2)\bi{2k}k}{(6k+1)\binom{3k}k \binom{6 k}{3 k}}={}&\int_0^1\left[\frac{t(1-t^2)}{2^{2}}\right]^{2 k}\log (1-t^2)\D t.\label{eq:Itab2c'}\end{align}}\end{lemma}\begin{proof}We can justify \eqref{eq:Itab2a}--\eqref{eq:Itab2c'} by considering Euler's beta integral $\mathrm B(a,b)=\int_0^1 t^{a-1}(1-t)^{b-1}\D t$ (where $ \RE a>0$ and $\RE b>0$) as well as its derivatives with respect to $a$ and $b$.
Here, one can express $ \mathrm B(a,b)=\frac{\Gamma(a)\Gamma(b)}{\Gamma(a+b)}$ through Euler's  gamma function $\Gamma(s) \colonequals\int_0^\infty t^{s-1}e^{-t}\D t$ for $ \RE s>0$, while  Euler's polygamma  function $ \psi^{(m)}(s)\colonequals\frac{\D^{m+1}}{\D s^{m+1}}\log\Gamma(s)$ satisfies $ \psi^{(0)}(k+1)=\psi^{(0)}(1)+\mathsf H_k$  and $ \psi^{(1)}(k+1)=\frac{\pi^{2}}{6}-\mathsf H_{k}^{(2)}$ for $ k\in\mathbb Z_{>0}$.
\end{proof}

\subsection{Generalized polylogarithms and multiple polylogarithms}While evaluating integral representations for the infinite series of our interest, we need to cope with integrands that involve logarithms and rational functions. This calls for the generalized polylogarithms (GPLs) $ G(\alpha_{1},\dots,\alpha_n;z)$, which satisfy the initial conditions \begin{align}
G(\underset{n }{\underbrace{0,\dots,0 }};z)\colonequals\frac{\log^nz}{n!},\quad G(-\!\!-;z)\colonequals1,\label{eq:GPL_bd}
\end{align}and can be constructed recursively \cite[\S2]{Frellesvig2016}  by\begin{align}
G(\alpha_{1},\dots,\alpha_n;z)\colonequals\int_0^z\frac{\D x}{x-\alpha_1}G(\alpha_2,\dots,\alpha_n;x)\label{eq:GPL_rec}
\end{align}    if    $(\alpha_1,\dots,\alpha_n)\neq\bm0\in\mathbb C^n$. Here,  the path of integration   in \eqref{eq:GPL_rec} runs along a straight line segment
from $0$ to $z$. Through Panzer's \texttt{HyperInt} package \cite{Panzer2015}, one may carry out heavy-duty symbolic manipulations of GPLs based on the recursion \eqref{eq:GPL_rec}.

  As in a companion paper \cite{SunZhou2024sum3k4k},
we will need the following  $ \mathbb Q$-vector space\begin{align}
\mathfrak G_{r;r_*}^{(z)}[S;S_*]\colonequals{}&\sum_{j=r_*}^r\Span_{\mathbb Q}\left\{ G(\alpha_1,\dots,\alpha_r;z)\left|\begin{smallmatrix}\alpha_\ell\in\{0,1\}\cup  S,\ell\in\mathbb Z\cap([1,r]\smallsetminus\{j\})\\{\alpha_j}\in{S_*}\end{smallmatrix}\right. \right\}
\end{align} for suitably chosen $ z\in\mathbb C$, $ r,r_*\in \mathbb Z_{>0}$, and $ S,S_*\subset\mathbb C$. Here, it is understood that each occurrence of  $ G(\alpha_1,\dots,\alpha_r;\alpha_{1})$ [being literally divergent by \eqref{eq:GPL_rec}] should be reinterpreted as its ``finite part'' according to Panzer's logarithmic regularization procedure \cite[\S2.3]{Panzer2015}.

The series defining Goncharov's \cite{Goncharov1997,Goncharov1998} multiple polylogarithms (MPLs)\begin{align}
\Li_{a_1,\dots,a_n}(z_1,\dots,z_n)\colonequals \sum_{\ell_{1}>\dots>\ell_{n}>0}\prod_{j=1}^n\frac{z_{j}^{\ell_{j}}}{\ell_j^{a_j}},
\label{eq:Mpl_defn}\end{align}converge absolutely for  $ a_1,\dots,a_n\in\mathbb Z_{>0}, \prod_{j=1}^m|z_j|<1,m\in\mathbb Z\cap[1,n]$. They admit analytic continuations through the following GPL-MPL\ correspondence (see \cite[(1.3)]{Panzer2015} or \cite[(2.5)]{Frellesvig2016})\begin{align} \Li_{a_1,\dots,a_n}(z_1,\dots,z_n)=(-1)^{n}G\left(\smash[b]{\underset{a_1-1 }{\underbrace{0,\dots,0 }}},\frac{1}{z_{1}},\smash[b]{\underset{a_2-1 }{\underbrace{0,\dots,0 }}},\frac{1}{z_{1}z_2},\dots,\smash[b]{\underset{a_n-1 }{\underbrace{0,\dots,0 }}},\frac{1}{\prod_{j=1}^nz_j};1\right),\label{eq:MPL_GPL}\\[-10pt]\notag\end{align}where
  $ \prod_{j=1}^nz_j\neq0$.

The  $ \mathbb Q$-vector space\begin{align}
\begin{split}
\mathfrak Z_{ r}(N)\colonequals{}&\Span_{\mathbb Q}\left\{\Li_{a_1,\dots,a_n}(z_1,\dots,z_n)\left|\begin{smallmatrix}a_1,\dots,a_n\in\mathbb Z_{>0}\\z_{1}^{N}=\dots=z_n^N=1\\(a_1,z_1)\neq(1,1)\\\sum _{j=1}^{n}a_{j}=r\end{smallmatrix}\right. \right\}\\\xlongequal{\text{\eqref{eq:MPL_GPL}}}{}&\Span_{\mathbb Q}\left\{G(z_1,\dots,z_r;1)\left|\begin{smallmatrix}z_1^N,\dots,z_{r}^N\in\{0,1\}\\z_1\neq1,z_{r}\neq0\end{smallmatrix}\right.\right\}
\end{split}\label{eq:Zk(N)_defn}
\end{align}is spanned by the cyclotomic multiple zeta values (CMZVs) of weight $s$ and level $N$. In this work, our closed-form evaluations of certain infinite series will be expressed through members of $ \mathfrak Z_r(8)$ and $ \mathfrak Z_r(12)$. Panzer's  \texttt{HyperInt} package \cite{Panzer2015} has native support for CMZVs of levels $ N\in\{1,2,4\}$ only, so we will use Au's \texttt{MultipleZetaValues} package \cite{Au2022a} for basis reductions of GPL expressions at levels  $ N\in\{8,12\}$.

\subsection{Some sets and functions related to cubic equations}
When $ x\notin\left\{ -\frac{2}{\sqrt{3}},-1,-\frac{1}{\sqrt{3}},0,\frac{1}{\sqrt{3}},1,\frac{2}{\sqrt{3}} \right\}$,   the set \begin{align}S_x\colonequals \left\{x,-x,\frac{-x+i\sqrt{3x^2-4}}{2},\frac{-x-i\sqrt{3x^2-4}}{2},\frac{x+i\sqrt{3x^2-4}}{2},\frac{x-i\sqrt{3x^2-4}}{2}\right\} \label{eq:Sx_defn}\end{align} always contains six distinct members.  To facilitate analysis, we also need its subsets \begin{align} S_x^+\colonequals {}&\left\{ x, \frac{-x+i\sqrt{3x^2-4}}{2},\frac{-x-i\sqrt{3x^2-4}}{2}\right\}\label{eq:Sx+defn}\intertext{and}  S_x^-\colonequals{}& \left\{ -x,\frac{x+i\sqrt{3x^2-4}}{2} ,\frac{x-i\sqrt{3x^2-4}}{2}\right\}.\label{eq:Sx-defn}\end{align} These sets are pertinent to the integral representations in \S\ref{subsec:6k_int_repn}, in that the cubic polynomial $t\big(1-t^2\big)$ evaluates to $ x\big(1-x^2\big)$ \big[resp.\ $ -x\big(1-x^2\big)$\big] when  $t\in S_x^+$ [resp.\ $ t\in S_x^-$].

With the notations \begin{align} \sigma_{\ell,m}(x)\colonequals\frac{(-1)^{\ell}x+i(-1)^m\sqrt{3x^2-4}}{2},\quad\ell,m\in\{0,1\}   \label{eq:sigma_defn}\end{align}and the dilogarithm  $ \Li_2(z)\colonequals -G\big(0,\frac1z;1\big)$   [cf.\ \eqref{eq:MPL_GPL}], we define four special functions {\allowdisplaybreaks\begin{align}
\begin{split}
\mathfrak{L}(x)\colonequals{}&2\left[ \Li_2\left(-\frac{1}{x}\right)-\Li_2\left(\frac{1}{x}\right) \right]-\sum _{\ell,m\in\{0,1\}} (-1)^{\ell }\Li_2\left(\frac{1}{\sigma_{\ell,m}(x) }\right)\\={}&3\left[ \Li_2\left(-\frac{1}{x}\right)-\Li_2\left(\frac{1}{x}\right) \right]+\left(\sum_{w\in S_x^+}-\sum_{w\in S_x^-}\right)\Li_2\left(\frac{1}{w}\right),
\end{split}\label{eq:L(x)defn}\\\mathfrak{M}(x)\colonequals {}&\sum _{\ell,m\in\{0,1\}} (-1)^{m }\Li_2\left(\frac{1}{\sigma_{\ell,m}(x) }\right),\\\begin{split}
\mathfrak{l}(x)\colonequals {}&2\left[\Li_2\left(\frac{1-x}{1+x}\right)-\Li_2\left(\frac{1+x}{1-x}\right)+\log ^2\left(1+\frac{1}{x}\right)-\log ^2\left(1-\frac{1}{x}\right)\right]\\{}&+\sum _{\ell,m\in\{0,1\}} (-1)^{\ell }\left[\Li_2\left( \frac{1-\sigma_{\ell,m}(x)}{1+\sigma_{\ell,m}(x)}\right) +\log^{2}\left(1+ \frac{1}{\sigma_{\ell,m}(x)} \right) \right]\\={}&3\left[\Li_2\left(\frac{1-x}{1+x}\right)-\Li_2\left(\frac{1+x}{1-x}\right)+\log ^2\left(1+\frac{1}{x}\right)-\log ^2\left(1-\frac{1}{x}\right)\right]\\{}&-\left(\sum_{w\in S_x^+}-\sum_{w\in S_x^-}\right)\left[\Li_2\left( \frac{1-w}{1+w}\right) +\log^{2}\left(1+ \frac{1}{w} \right) \right],
\end{split}\label{eq:l(x)defn}\\\mathfrak{m}(x)\colonequals {}&\sum _{\ell,m\in\{0,1\}} (-1)^{m }\left[\Li_2\left( \frac{1-\sigma_{\ell,m}(x)}{1+\sigma_{\ell,m}(x)}\right) +\log^{2}\left(1+ \frac{1}{\sigma_{\ell,m}(x)} \right) \right].\label{eq:m(x)defn}\end{align}}These functions will allow us to compress many closed-form evaluations for infinite series later in \S\ref{sec:polylog6k}.

\begin{lemma}\label{lm:LiGPL}\begin{enumerate}[leftmargin=*,  label=\emph{(\alph*)},ref=(\alph*),
widest=d, align=left] \item
For $ t\in(0,1)$, $ \big|x\big(1-x^{2}\big)\big|>\frac{2}{3\sqrt{3}}$, and $ r\in\mathbb Z_{>0}$, one has\begin{align}
\begin{split}
\Li_{r+1}\left(\left[ \frac{t\big(1-t^{2}\big)}{x\big(1-x^{2}\big)} \right]^{2} \right)={}&-2^r\sum_{\alpha_1,\dots,\alpha_r\in\{-1,0,1\}}\sum_{w\in S_x}G(\alpha_{1},\dots,\alpha_r,w;t)\\\in{}&\mathfrak{G}^{(t)}_{r+1;r+1}\left[\{-1\};\left\{ x,-x,\sigma_{0,0}(x),\sigma_{0,1}(x),\sigma_{1,0}(x),\sigma_{1,1}(x)\right\} \right].
\end{split}\label{eq:LiSx}
\end{align}
\item For $ t\in(0,1)$, $ \big|x\big(1-x^{2}\big)\big|>\frac{2}{3\sqrt{3}}$, and $ r\in\mathbb Z_{>0}$, one has\begin{align}
\begin{split}
\Li_{r+1}\left(\frac{t\big(1-t^{2}\big)}{x\big(1-x^{2}\big)} \right)={}&-\sum_{\alpha_1,\dots,\alpha_r\in\{-1,0,1\}}\sum_{w\in S^+_x}G(\alpha_{1},\dots,\alpha_r,w;t)\\\in{}&\mathfrak{G}^{(t)}_{r+1;r+1}\left[\{-1\};\left\{ x,\sigma_{1,0}(x),\sigma_{1,1}(x)\right\} \right]
\end{split}\label{eq:LiSx+}\intertext{and}\begin{split}
\Li_{r+1}\left(-\frac{t\big(1-t^{2}\big)}{x\big(1-x^{2}\big)} \right)={}&-\sum_{\alpha_1,\dots,\alpha_r\in\{-1,0,1\}}\sum_{w\in S^-_x}G(\alpha_{1},\dots,\alpha_r,w;t)\\\in{}&\mathfrak{G}^{(t)}_{r+1;r+1}\left[\{-1\};\left\{ -x,\sigma_{0,0}(x),\sigma_{0,1}(x)\right\} \right].
\end{split}\label{eq:LiSx-}
\end{align}
\end{enumerate}\begin{proof}\begin{enumerate}[leftmargin=*,  label=(\alph*),ref=(\alph*),
widest=d, align=left] \item For $ t\in(0,1)$ and $|x|$ sufficiently  large, we can check directly that\begin{align}
\Li_{1}\left(\left[ \frac{t\big(1-t^{2}\big)}{x\big(1-x^{2}\big)} \right]^{2} \right)=-\log\left( 1-\left[ \frac{t\big(1-t^{2}\big)}{x\big(1-x^{2}\big)} \right]^{2} \right)=-\sum_{w\in S_x}G(w;t).
\end{align}
The same identity analytically continues\footnote{Throughout this article, we often write $ \log P(t,x)=\sum_r(-1)^{n_{r}}\log(1-f_r(t,x))$  when $ P(t,x)=\prod_r[1-f_r(t,x)]^{(-1)^{n_r}}$ holds for $ n_r\in\{0,1\}$, $t\in(0,1) $, and $ \lim_{x\to\infty} P(t,x)=1$. Such manipulations of complex-valued logarithms are readily justified when $|x|$ is sufficiently large, while the resulting formulae admit analytic continuations to all the points in the region    $ \big|x\big(1-x^{2}\big)\big|>\frac{2}{3\sqrt{3}}$. Hereafter, we will perform such decompositions of natural logarithms without reiterating the aforementioned procedures for their justification.}   to all $x$ satisfying  $ \big|x\big(1-x^{2}\big)\big|>\frac{2}{3\sqrt{3}}$, where $ G(w;t)=-\log\left( 1-\frac{t}{w} \right)$ is free from branch cuts for any $ w\in S_x$ \big[N.B.: one has  $w\notin(-1,1)$ when $\big|x\big(1-x^{2}\big)\big|>\frac{2}{3\sqrt{3}}$\big]. Integrating  over $ \frac{\D}{\D z}\Li_{r+1}(z)=\frac{\Li_{r}(z)}{z}$, we have\begin{align}
\Li_{r+1}\left(\left[ \frac{t\big(1-t^{2}\big)}{x\big(1-x^{2}\big)} \right]^{2} \right)=2\int_0^t\left(\frac{1}{u+1} +\frac{1}{u}+\frac{1}{u-1}\right)\Li_r\left( \left[ \frac{u\big(1-u^{2}\big)}{x\big(1-x^{2}\big)} \right]^{2} \right)\D u.
\end{align}Thus, one can prove \eqref{eq:LiSx} by inductive applications of the GPL recursion in \eqref{eq:GPL_rec}.
\item Start from \begin{align}
\Li_{1}\left(\pm \frac{t\big(1-t^{2}\big)}{x\big(1-x^{2}\big)} \right)=-\sum_{w\in S^\pm_x}G(w;t)
\end{align}and build inductively on \begin{align}
\Li_{r+1}\left(\pm \frac{t\big(1-t^{2}\big)}{x\big(1-x^{2}\big)} \right)=\int_0^t\left(\frac{1}{u+1} +\frac{1}{u}+\frac{1}{u-1}\right)\Li_r\left( \pm \frac{t\big(1-t^{2}\big)}{x\big(1-x^{2}\big)} \right)\D u,
\end{align}while invoking the recursive definition   \eqref{eq:GPL_rec} of GPLs.   \qedhere\end{enumerate}
\end{proof}
\end{lemma}
\section{Polylogarithmic reductions of certain series involving $ \frac{ \binom{2k}{k}}{\binom{3k}{k}\binom{6k}{3k}}$\label{sec:polylog6k}}

\subsection{Polylogarithmic characterizations\label{subsec:Li_n_char}}We begin with a variation on \cite[Theorem 1.5(a) and Theorem 1.6(a)]{SunZhou2024sum3k4k}.\begin{theorem}\label{thm:6kGPL}Recall $ \sigma_{\ell,m}(x)$ from \eqref{eq:sigma_defn}. If $ \big|x\big(1-x^{2}\big)\big|>\frac{2}{3\sqrt{3}}$ and $ r\in\mathbb Z_{>0}$, then we have\begin{align}&
\sum_{k=1}^\infty\frac{a_k\binom{2k}k}{k^{r+1}\binom{3k}k\binom{6k}{3k}}\left[ \frac{2^{2}}{x\big(1-x^{2}\big)} \right]^{2k}\in{\mathfrak{G}}^{(1)}_{r+2;r+1}\left[\{-1\};\left\{ x,-x,\sigma_{0,0}(x),\sigma_{0,1}(x),\sigma_{1,0}(x),\sigma_{1,1}(x)\right\} \right],\label{eq:ak_sum6k}
\end{align}where  $ ( a_k)_{k\in\mathbb Z_{>0}}\in\big\{\big(\frac1k\big)_{k\in\mathbb Z_{>0}},$  $(-\mathsf H_k+2\mathsf  H_{2 k}+\mathsf H_{3 k}-2\mathsf  H_{6 k})_{k\in\mathbb Z_{>0}},$ $(\mathsf H_{2 k-1}+\mathsf H_{3 k}-2\mathsf H_{6 k}+2\log2)_{k\in\mathbb Z_{>0}}\big\}$. If $ \big|x\big(1-x^{2}\big)\big|>\frac{2}{3\sqrt{3}}$ and $r,r_{1},\dots, r_M\in\mathbb Z_{>0} $,  then we have\begin{align}\begin{split}&
\sum_{k=1}^\infty\left[ \prod_{j=1}^M\mathsf H_{k}^{(r_j)} \right]\frac{\binom{2k}k}{k^{r+1}\binom{3k}k\binom{6k}{3k}}\left[ \frac{2^{2}}{x\big(1-x^{2}\big)} \right]^{2k}\\\in{}&{\mathfrak{G}}^{(1)}_{r+1+\sum_{j=1}^M r_j;r+1+\sum_{j=1}^M r_j}\left[\left\{ -1,x,-x,\sigma_{0,0}(x),\sigma_{0,1}(x),\sigma_{1,0}(x),\sigma_{1,1}(x)\right\} ;\right.\\&{}\left.\left\{ x,-x,\sigma_{0,0}(x),\sigma_{0,1}(x),\sigma_{1,0}(x),\sigma_{1,1}(x)\right\} \right].\label{eq:Hk_sum6k}
\end{split}\end{align}
\end{theorem}\begin{proof}In view of \eqref{eq:Itab2a}--\eqref{eq:Itab2c}, we have{\allowdisplaybreaks
\begin{align}
\sum_{k=1}^\infty\frac{\binom{2k}k}{k^{r+2}\binom{3k}k\binom{6k}{3k}}\left[ \frac{2^{2}}{x\big(1-x^{2}\big)} \right]^{2k}={}&4\int_0^1   \frac{\Li_{r+1}\left(\left[\frac{t\big(1-t^{2}\big)}{x\big(1-x^{2}\big)}\right] ^{2}\right)\D t}{1-t^{2}},\\\sum_{k=1}^\infty\frac{(-\mathsf H_k+2\mathsf  H_{2 k}+\mathsf H_{3 k}-2\mathsf  H_{6 k})\binom{2k}k}{k^{r+1}\binom{3k}k\binom{6k}{3k}}\left[ \frac{2^{2}}{x\big(1-x^{2}\big)} \right]^{2k}={}&8\int_0^1   \frac{\Li_{r}\left(\left[\frac{t\big(1-t^{2}\big)}{x\big(1-x^{2}\big)}\right] ^{2}\right)\log t\D t}{1-t^{2}},\\\sum_{k=1}^\infty\frac{(\mathsf H_{2 k-1}+\mathsf H_{3 k}-2\mathsf H_{6 k}+2\log2)\binom{2k}k}{k^{r+1}\binom{3k}k\binom{6k}{3k}}\left[ \frac{2^{2}}{x\big(1-x^{2}\big)} \right]^{2k}={}&4\int_0^1   \frac{\Li_{r}\left(\left[\frac{t\big(1-t^{2}\big)}{x\big(1-x^{2}\big)}\right] ^{2}\right)\log\left( 1-t^{2} \right)\D t}{1-t^{2}},
\end{align}}where $ \Li_r(z)\colonequals \sum_{k=1}^\infty\frac{z^k}{k^r}$ defines the polylogarithm of order  $ r\in\mathbb Z_{>0}$ for $|z|<1$. From the proof of Lemma \ref{lm:LiGPL}(a), we know that \begin{align}
\Li_r\left(\left[\frac{t\big(1-t^{2}\big)}{x\big(1-x^{2}\big)}\right] ^{2}\right)\in{}& \mathfrak{G}^{(t)}_{r;r}\left[\{-1\};\left\{ x,-x,\sigma_{0,0}(x),\sigma_{0,1}(x),\sigma_{1,0}(x),\sigma_{1,1}(x)\right\} \right]
\end{align} holds for all $ r\in\mathbb Z_{>0}$.
Rewriting products of GPLs through their  shuffle algebra \cite[(2.4)]{Frellesvig2016}, we get\begin{align}
\Li_{r}\left(\left[\frac{t\big(1-t^{2}\big)}{x\big(1-x^{2}\big)}\right] ^{2}\right)\log t\in {}&\mathfrak{G}^{(t)}_{r+1;r}\left[\{-1\};\left\{ x,-x,\sigma_{0,0}(x),\sigma_{0,1}(x),\sigma_{1,0}(x),\sigma_{1,1}(x)\right\} \right],\\\Li_{r}\left(\left[\frac{t\big(1-t^{2}\big)}{x\big(1-x^{2}\big)}\right] ^{2}\right)\log\big( 1-t^{2} \big)\in {}&\mathfrak{G}^{(t)}_{r+1;r}\left[\{-1\};\left\{ x,-x,\sigma_{0,0}(x),\sigma_{0,1}(x),\sigma_{1,0}(x),\sigma_{1,1}(x)\right\} \right],
\end{align}for all $ r\in\mathbb Z_{>0}$.
We can turn the last three displayed formulae into \eqref{eq:ak_sum6k}, after appealing to the GPL recursion in \eqref{eq:GPL_rec}.

It can be inferred from \cite[Theorem 3.1]{Zhou2022mkMpl} that \begin{align}
\sum_{k=1}^\infty\frac{\prod_{j=1}^M\mathsf H_{k}^{(r_{j})}}{k^{r}}z^k\in\mathfrak G_{w;w}^{(z)}[\varnothing;\{1\}]\label{eq:H_prod_gf}
\end{align}for $ |z|<1$ and $ r\in\mathbb Z_{>0}$, so  our proof of Lemma \ref{lm:LiGPL}(a) can be adapted to\begin{align}\begin{split}
&\sum_{k=1}^\infty\frac{\prod_{j=1}^M\mathsf H_{k}^{(r_{j})}}{k^{r}}\left[\frac{t\big(1-t^{2}\big)}{x\big(1-x^{2}\big)}\right]^{2k}\\\in{}&\mathfrak G_{r+\sum_{j=1}^Mr_j;r+\sum_{j=1}^Mr_j}^{(t)}\left[\left\{ -1,x,-x,\sigma_{0,0}(x),\sigma_{0,1}(x),\sigma_{1,0}(x),\sigma_{1,1}(x)\right\};\right.\\&{}\left.\left\{ x,-x,\sigma_{0,0}(x),\sigma_{0,1}(x),\sigma_{1,0}(x),\sigma_{1,1}(x)\right\} \right].
\end{split}\end{align}Thus we get the statement in  \eqref{eq:Hk_sum6k}, upon yet another invocation of the GPL recursion \eqref{eq:GPL_rec}.
\end{proof}

\begin{table}[b]\caption{Selected CMZV characterizations of  $\mathsf S_{6,r}\big(a_k;2^{3}\big)$ at level $8$, where $ \theta\colonequals e^{\pi i/4}$, $\lambda\colonequals\log2 $, $\widetilde \lambda\colonequals\log\big(1+\sqrt{2}\big)$, and $ G\colonequals\I\Li_2(i)$\label{tab:Zk(8)6k}}

\begin{scriptsize}\begin{align*}\begin{array}{c@{}l}\hline\hline
&\vphantom{\frac{\int}{\int}}\mathsf S_{6,0}\big(\frac1k;2^3\big)=4\int_{0}^1\Li_1\left( \frac{[t(1-t^2)]^2}{2} \right)\frac{\D t}{1-t^2}\\{}={}&\frac{\pi ^2}{2}-6 \widetilde{\lambda }^2\\[5pt]&\vphantom{\frac{\int}{\int}}\mathsf S_{6,1}\big(\frac1k;2^3\big)=4\int_{0}^1\Li_2\left( \frac{[t(1-t^2)]^2}{2} \right)\frac{\D t}{1-t^2}\\{}={}&-144 \RE\left[\Li_{1,1,1}(i,1,\theta )+\Li_{1,1,1}(i,1,-\theta )\right]-112 \zeta (3)-144 \widetilde{\lambda } \RE\Li_{1,1}(i,\theta )+16 \pi  \I\Li_2(\theta )+14 \pi  G\\&{} -\frac{3\big(-18 \lambda ^2 \widetilde{\lambda }+12 \lambda  \widetilde{\lambda }^2-24 \widetilde{\lambda }^3+7 \lambda ^3\big)}{4}+\frac{3\pi ^2 \big(21 \lambda -2 \widetilde{\lambda }\big)}{16} \\[3pt]\hline& \vphantom{\frac{\int}{\int}}\mathsf S_{6,1}\big(\!-\mathsf H_k+2\mathsf  H_{2 k}+\mathsf H_{3 k}-2\mathsf  H_{6 k};2^3\big)=8\int_{0}^1\Li_1\left( \frac{[t(1-t^2)]^2}{2} \right)\frac{\log t\D t}{1-t^2}\\{}={}&64 \RE\left[\Li_{1,1,1}(i,1,\theta )+\Li_{1,1,1}(i,1,-\theta )\right]+32 \widetilde{\lambda } \RE\Li_{1,1}(i,\theta )+8 \pi  \I\Li_2(\theta )-10 \pi  G\\&{}-\frac{9 \lambda ^2 \widetilde{\lambda }+12 \widetilde{\lambda }^3-7 \lambda ^3}{3} -\frac{\pi ^2 \big(13 \lambda -25 \widetilde{\lambda }\big)}{12} \\[3pt]\hline &\vphantom{\frac{\frac\int\int}{\int}}\mathsf S_{6,1}\big(\mathsf H_{2 k-1}+\mathsf H_{3 k}-2\mathsf H_{6 k};2^3\big)=4\int_{0}^1\Li_1\left( \frac{[t(1-t^2)]^2}{2} \right)\frac{\log \frac{1-t^{2}}{2^{2}}\D t}{1-t^2}\\{}={}&-176\RE\left[\Li_{1,1,1}(i,1,\theta )+\Li_{1,1,1}(i,1,-\theta )\right]-112 \zeta (3)-88 \widetilde{\lambda } \RE\Li_{1,1}(i,\theta )+4 \pi  \I\Li_2(\theta )+21 \pi  G\\&{} +\frac{11\big(9 \lambda ^2 \widetilde{\lambda }+12 \widetilde{\lambda }^3-7 \lambda ^3\big)}{12} +\frac{11\pi ^2 \big(11 \widetilde{\lambda }+13 \lambda \big)}{48}\\[3pt]\hline\hline
\end{array}\end{align*}

\begin{align*}\begin{array}{c|l|l|l}\hline\hline
r&\vphantom{\frac{\int}{\int}}\mathsf S_{6,r}\big(\frac1k;2^3\big)&\mathsf S_{6,r+1}\big(\!-\mathsf H_k+2\mathsf  H_{2 k}+\mathsf H_{3 k}-2\mathsf  H_{6 k};2^3\big)&\mathsf S_{6,r+1}\big(\mathsf H_{2 k-1}+\mathsf H_{3 k}-2\mathsf H_{6 k};2^3\big)\\\hline -1\vphantom{\frac{\frac1\int}{1}}&-\frac{6 \sqrt{2} \widetilde{\lambda }}{5}+\frac{2 \sqrt{2} \pi }{5}&-\frac{16\sqrt{2} \left[\RE\Li_{1,1}(i,\theta )+\I\Li_2(\theta )\right]}{5}+\frac{4 \sqrt{2} G}{5} &\frac{4\sqrt{2} \left[11 \RE\Li_{1,1}(i,\theta )-2 \I\Li_2(\theta )\right]}{5}+\frac{2 \sqrt{2} G}{5} \\&&{}-\frac{4 \lambda  \widetilde{\lambda }-4 \widetilde{\lambda }^2-3 \lambda ^2}{5 \sqrt{2}}+\frac{23 \pi ^2}{60 \sqrt{2}}&{}+\frac{11 \big(4 \lambda  \widetilde{\lambda }-4 \widetilde{\lambda }^2-3 \lambda ^2\big)}{20 \sqrt{2}}-\frac{3\sqrt{2} \pi  \lambda}{5}  +\frac{143 \pi ^2}{240 \sqrt{2}}\\[8pt]-2&-\frac{21 \sqrt{2} \widetilde{\lambda }}{125} +\frac{12   \sqrt{2}\pi}{125}+\frac{2}{25} &-\frac{8\sqrt{2}\left[ 7 \RE\Li_{1,1}(i,\theta )+12 \I\Li_2(\theta ) \right]}{125} +\frac{24 \sqrt{2} G}{125}&\frac{2\sqrt{2} \left[77 \RE\Li_{1,1}(i,\theta )-24 \I\Li_2(\theta )\right]}{125}+\frac{12 \sqrt{2} G}{125} \\[2pt]&&{}-\frac{28 \lambda  \widetilde{\lambda }-28 \widetilde{\lambda }^2-240 \widetilde{\lambda }-21 \lambda ^2}{250 \sqrt{2}}+\frac{161 \pi ^2}{3000 \sqrt{2}}&{}+\frac{308 \lambda  \widetilde{\lambda }-308 \widetilde{\lambda }^2+720 \widetilde{\lambda }-231 \lambda ^2}{1000 \sqrt{2}}+\frac{1001 \pi ^2}{12000 \sqrt{2}}\\&&{}-\frac{3 \sqrt{2} \pi }{25}&{}-\frac{\pi  (36 \lambda +35)}{125 \sqrt{2}}\\[8pt]-3&-\frac{51 \widetilde{\lambda }}{3125 \sqrt{2}}+\frac{297 \pi }{3125 \sqrt{2}}+\frac{81}{625}&-\frac{4 \sqrt{2} \left[17 \RE\Li_{1,1}(i,\theta )+297 \I\Li_2(\theta )\right]}{3125}+\frac{297 \sqrt{2} G}{3125}&\frac{11 \sqrt{2} \left[17 \RE\Li_{1,1}(i,\theta )-54 \I\Li_2(\theta )\right]}{3125}+\frac{297 G}{3125 \sqrt{2}}\\&&{}-\frac{17 \big(4 \lambda  \widetilde{\lambda }-4 \widetilde{\lambda }^2-320 \widetilde{\lambda }-3 \lambda ^2\big)}{12500 \sqrt{2}}+\frac{391 \pi ^2}{150000 \sqrt{2}}&{}+\frac{748 \lambda  \widetilde{\lambda }-748 \widetilde{\lambda }^2+5920 \widetilde{\lambda }-561 \lambda ^2}{50000 \sqrt{2}}+\frac{2431 \pi ^2}{600000 \sqrt{2}}\\&&{}-\frac{34 \sqrt{2} \pi }{625}&{}-\frac{\pi  (891 \lambda +860)}{6250 \sqrt{2}}-\frac{1}{25}\\[5pt]\hline\hline
\end{array}\end{align*}
\end{scriptsize}\end{table}

\begin{table}[h]\caption{Selected CMZV characterizations of  $\mathsf  S_{6,r}\big(a_k;2^{2}3^3\big)$ at level $12$, where  $ \omega\colonequals e^{2\pi i/3}$,  $ \varrho\colonequals e^{\pi i/3}$, $ \lambda\colonequals\log2$, $ \widetilde\varLambda\colonequals\log\big(2+\sqrt{3}\big)$, and $ G\colonequals\I\Li_2(i)$\label{tab:Zk(12)6k}}

\begin{scriptsize}\begin{align*}\begin{array}{@{}c@{}l}\hline\hline
&\vphantom{\frac{\int}{\int}}\mathsf S_{6,0}\big(\frac1k;2^{2}3^3\big)=4\int_{0}^1\Li_1\left( \frac{3^{3}[t(1-t^2)]^2}{2^{2}} \right)\frac{\D t}{1-t^2}\\{}={}&2 \pi ^2-6 \widetilde{\varLambda }^2\\[5pt]&\mathsf S_{6,1}\big(\frac1k;2^{2}3^3\big)=4\int_{0}^1\Li_2\left( \frac{3^{3}[t(1-t^2)]^2}{2^{2}} \right)\frac{\D t}{1-t^2}\\{}={}&96 \RE\left[\Li_{1,1,1}\left(\varrho ,1,\frac{i}{\varrho }\right)+\Li_{1,1,1}\left(\varrho ,1,-\frac{i}{\varrho }\right)\right]-72 \RE\left[\Li_{1,1,1}\left(\omega ,1,\frac{i}{\varrho }\right)+\Li_{1,1,1}\left(\omega ,1,-\frac{i}{\varrho }\right)\right]+\frac{31 \zeta (3)}{2}\\&{}+96 \widetilde{\varLambda } \RE\Li_{1,1}\left(\varrho ,\frac{i}{\varrho }\right)-72 \widetilde{\varLambda } \RE\Li_{1,1}\left(\omega ,\frac{i}{\varrho }\right)-15 \pi  \I\Li_2(\omega ) +3 \widetilde{\varLambda }^2 \big(3 \widetilde{\varLambda }-4 \lambda -3 \varLambda \big)-\frac{\pi ^2 \big(19 \widetilde{\varLambda }-48 \lambda-87 \varLambda \big)}{12}\\[3pt]\hline& \vphantom{\frac{\int}{\int}}\mathsf S_{6,1}\big(\!-\mathsf H_k+2\mathsf  H_{2 k}+\mathsf H_{3 k}-2\mathsf  H_{6 k};2^{2}3^3\big)=8\int_{0}^1\Li_1\left( \frac{3^{3}[t(1-t^2)]^2}{2^{2}} \right)\frac{\log t\D t}{1-t^2}\\{}={}&-128 \RE\left[\Li_{1,1,1}\left(\varrho ,1,\frac{i}{\varrho }\right)+\Li_{1,1,1}\left(\varrho ,1,-\frac{i}{\varrho }\right)\right]-48 \RE\left[\Li_{1,1,1}\left(\omega ,1,\frac{i}{\varrho }\right)+\Li_{1,1,1}\left(\omega ,1,-\frac{i}{\varrho }\right)\right]-\frac{580 \zeta (3)}{3}\\&{}-24 \widetilde{\varLambda } \RE\Li_{1,1}\left(\omega ,\frac{i}{\varrho }\right)-64 \widetilde{\varLambda } \RE\Li_{1,1}\left(\varrho ,\frac{i}{\varrho }\right)+38 \pi  \I\Li_2(\omega )+3 \widetilde{\varLambda }^3+\frac{\pi ^2 \big(281 \widetilde{\varLambda }+30 \varLambda \big)}{36}\\[3pt]\hline &\vphantom{\frac{\frac\int\int}{\int}}\mathsf S_{6,1}\big(\mathsf H_{2 k-1}+\mathsf H_{3 k}-2\mathsf H_{6 k};2^{2}3^3\big)=4\int_{0}^1\Li_1\left( \frac{3^{3}[t(1-t^2)]^2}{2^{2}} \right)\frac{\log \frac{1-t^{2}}{2^{2}}\D t}{1-t^2}\\{}={}& 160 \RE\left[\Li_{1,1,1}\left(\varrho ,1,\frac{i}{\varrho }\right)+\Li_{1,1,1}\left(\varrho ,1,-\frac{i}{\varrho }\right)\right]-48 \RE\left[\Li_{1,1,1}\left(\omega ,1,\frac{i}{\varrho }\right)+\Li_{1,1,1}\left(\omega ,1,-\frac{i}{\varrho }\right)\right]+\frac{673 \zeta (3)}{6}\\&{}-24 \widetilde{\varLambda } \RE\Li_{1,1}\left(\omega ,\frac{i}{\varrho }\right)+80 \widetilde{\varLambda } \RE\Li_{1,1}\left(\varrho ,\frac{i}{\varrho }\right)-34 \pi  \I\Li_2(\omega )+3 \widetilde{\varLambda }^3-\frac{\pi ^2 \big(169 \widetilde{\varLambda }-30 \varLambda\big)}{36} \\[3pt]\hline& \vphantom{\frac{\int}{\int}}\mathsf S_{6,1}\big(\mathsf H_k;2^{2}3^3\big)=4\int_{0}^1\left[\Li_2\left( \frac{3^{3}[t(1-t^2)]^2}{2^{2}} \right)+\frac{1}{2}\log^{2}\left( 1-\frac{3^{3}[t(1-t^2)]^2}{2^{2}} \right)\right]\frac{\D t}{1-t^2}\\{}={}&288 \RE\left[\Li_{1,1,1}\left(\varrho ,1,\frac{i}{\varrho }\right)+\Li_{1,1,1}\left(\varrho ,1,-\frac{i}{\varrho }\right)\right]-216 \RE\left[\Li_{1,1,1}\left(\omega ,1,\frac{i}{\varrho }\right)+\Li_{1,1,1}\left(\omega ,1,-\frac{i}{\varrho }\right)\right]+\frac{457 \zeta (3)}{2}\\&{}-144 \widetilde{\varLambda } \RE\Li_{1,1}\left(\omega ,\frac{i}{\varrho }\right)+192 \widetilde{\varLambda } \RE\Li_{1,1}\left(\varrho ,\frac{i}{\varrho }\right)-45 \pi  \I\Li_2(\omega )-3 \widetilde{\varLambda }^2 \big(\!-6 \widetilde{\varLambda }+4 \lambda +3 \varLambda \big)+\frac{\pi ^2 \big(-182 \widetilde{\varLambda }+48 \lambda +117 \varLambda \big)}{12} \\[3pt]\hline\hline
\end{array}\end{align*}
\end{scriptsize}\end{table}
\begin{corollary}
Set \begin{align}
\mathsf S_{6,r}(a_k;z)\colonequals{}&\sum_{k=1}^\infty\frac{a_{k}\binom{2k}kz^k}{k^{r+1}\binom{3k}k\binom{6k}{3k}}.
\end{align}If  $r\in\mathbb Z_{>0} $,  then we have\begin{align}
\mathsf S_{6,r}\big(a_k;2^{3}\big)\in{}&\mathfrak Z_{r+2}(8),\label{eq:S6Z(8)}
\intertext{when $ (a_k)_{k\in\mathbb Z_{>0}}\in\big\{\big(\frac1k\big)_{k\in\mathbb Z_{>0}},$  $(-\mathsf H_k+2\mathsf  H_{2 k}+\mathsf H_{3 k}-2\mathsf  H_{6 k})_{k\in\mathbb Z_{>0}},$ $(\mathsf H_{2 k-1}+\mathsf H_{3 k}-2\mathsf H_{6 k})_{k\in\mathbb Z_{>0}}\big\}$, as well as }\mathsf S_{6,r}\big(a_k;2^23^{3}\big)\in{}&\mathfrak Z_{r+2}(12),\label{eq:S6Z(12)}\end{align}when $ (a_k)_{k\in\mathbb Z_{>0}}\in\big\{\big(\frac1k\big)_{k\in\mathbb Z_{>0}},$  $(\mathsf H_{3 k}-2\mathsf H_{6 k})_{k\in\mathbb Z_{>0}},$ $(\mathsf H_{2 k})_{k\in\mathbb Z_{>0}},$ $(\mathsf H_k)_{k\in\mathbb Z_{>0}}\big\}$.\end{corollary}\begin{proof}In Au's   \texttt{MultipleZetaValues} (v1.2.0) package \cite{Au2022a},
 if  \texttt{IterIntDoableQ[\textbraceleft0,1\textbraceright$\cup S\cup $\textbraceleft$ z$\textbraceright]} produces a positive integer that  divides $N$ for every point $ z\in S_*$, then we have $ \mathfrak G_{r,r_*}^{(1)}[S,S_*]\subseteq\mathfrak Z_r(N) $. Running Au's \texttt{IterIntDoableQ} test on \eqref{eq:ak_sum6k} and \eqref{eq:Hk_sum6k}, we get \eqref{eq:S6Z(8)} and \eqref{eq:S6Z(12)}. \big[Here, for the borderline case \eqref{eq:S6Z(12)}, one exploits Theorem \ref{thm:6kGPL} in a regime where $\big|x\big(1-x^{2}\big)\big|$ approaches $\frac{2}{3\sqrt{3}}$ from above.\big]
\end{proof}
\begin{remark}In Tables \ref{tab:Zk(8)6k} and \ref{tab:Zk(12)6k}, we illustrate the corollary above for $r=1$. In addition, if  we retroactively define $ \Li_r(u)\colonequals\sum_{k=1}^\infty\frac{u^k}{k^r}$ with $|u|<1$  for $ r\in\mathbb Z_{\leq0}$, then  formulae like\begin{align}
\mathsf S_{6,r}\left(\frac1k;z\right)=4\int_0^1\Li_r\left( z\left[\frac{t(1-t^{2})}{4}\right]^2 \right)\frac{\D t}{1-t^2}\label{eq:retroLi}
\end{align}remain valid for all  $ r\in\mathbb Z$ and $ |z|<2^{2}3^3$. Some of these cases are  listed in Table \ref{tab:Zk(8)6k}, as support for \cite[Conjectures 5.4(i) and 5.5(i)]{Sun2023}, namely \eqref{eq:conj5.4(1)}--\eqref{eq:conj5.5(2)}. Here, we note that the constant  $ L\colonequals\sum_{k=0}^\infty\f{(-1)^{k(k-1)/2}}{(2k+1)^2}$ satisfies $ L=\sqrt{2}\left[\I \Li_2(\theta)-\frac{G}{4}\right]$.
\eor\end{remark}\begin{remark}Notably, the summation formulae \eqref{eq:Remark5.5a} and \eqref{eq:Remark5.5b}  are two interesting combinatorial consequences of \eqref{eq:conj5.4(1)} that do not require any further integrations. 

By induction,
 \begin{align}\sum_{k=1}^n\f{(100k^2-112k+15)\bi{2k}k8^k}{k\bi{3k}k\bi{6k}{3k}}
 =4-4(2n+1)\f{8^n\bi{2n}n}{\bi{3n}n\bi{6n}{3n}}\end{align}
 for all $n=1,2,3,\ldots$. It follows that
 \begin{align}\sum_{k=1}^\infty\f{(100k^2-112k+15)\bi{2k}k8^k}{k\bi{3k}k\bi{6k}{3k}}=4.\end{align}
 Since
 $2k(350k-17)=7(100k^2-112k+15)+15(50k-7)$, we have
 \begin{align}\begin{split}
15\sum_{k=1}^\infty\f{(50k-7)\bi{2k}k8^k}{k\bi{3k}k\bi{6k}{3k}}
 =&\ 2\sum_{k=1}^\infty\f{(350k-17)\bi{2k}k8^k}{\bi{3k}k\bi{6k}{3k}}
 -7\sum_{k=1}^\infty\f{(100k^2-112k+15)\bi{2k}k8^k}{k\bi{3k}k\bi{6k}{3k}}
 \\=&\ 2(15\sqrt2\,\pi+27+17)-7\times4=30\sqrt2\,\pi+60
\end{split}
 \end{align}
 and hence \eqref{eq:Remark5.5a}. The same identity can also be verified by combining certain entries of Table \ref{tab:Zk(8)6k}.

 By induction, for $n=1,2,3,\ldots$ we have
 \begin{align}\sum_{k=1}^n\f{(100k^2-104k+15)\bi{2k}k8^k}{k(2k-1)\bi{3k}k\bi{6k}{3k}}=4-4\times\f{8^n\bi{2n}n}{\bi{3n}n\bi{6n}{3n}}.\end{align}
 Letting $n\to+\infty$, we obtain
\begin{align}\sum_{k=1}^\infty\f{(100k^2-104k+15)\bi{2k}k8^k}{k(2k-1)\bi{3k}k\bi{6k}{3k}}=4.\end{align}
 As $(2k-1)(50k-7)=100k^2-14k+15+8(5k-1),$ we have
 \begin{align}\begin{split}
8\sum_{k=1}^\infty\f{(5k-1)\bi{2k}k8^k}{k(2k-1)\bi{3k}k\bi{6k}{3k}}
 =&\ \sum_{k=1}^\infty\f{(50k-7)\bi{2k}k8^k}{k\bi{3k}k\bi{6k}{3k}}-\sum_{k=1}^\infty\f{(100k^2-104k+15)\bi{2k}k8^k}{k(2k-1)\bi{3k}k\bi{6k}{3k}}
 \\=&\ 2\sqrt2\,\pi+4-4=2\sqrt2\,\pi
\end{split}
 \end{align}
 and hence \eqref{eq:Remark5.5b}. Unlike the last paragraph, we cannot deduce \eqref{eq:Remark5.5b} from Table \ref{tab:Zk(8)6k} alone, but will need entry (D.\DDref{eq:sum6(2k-1)}) in Table \ref{tab:6kLi2} below to complete the circle.
\eor\end{remark}

\subsection{Dilogarithmic evaluations}In this subsection, we derive analogs of  \cite[Theorem 1.5(b) and Theorem 1.6(b)]{SunZhou2024sum3k4k}. With these analogs, we prove  \cite[Conjecture 5.5(ii)]{Sun2023}, namely \eqref{eq:conj5.5ii1}--\eqref{eq:conj5.5ii3}.
\begin{proposition}Recall the functions $ \mathfrak L(x)$, $ \mathfrak M(x)$, $ \mathfrak l(x)$, and $ \mathfrak m(x)$ from \eqref{eq:L(x)defn}--\eqref{eq:m(x)defn}.

 If $ \big|x\big(1-x^{2}\big)\big|>\frac{2}{3\sqrt{3}}$, then we have the series evaluations in Table \ref{tab:6kLi2}, expressible through logarithms and dilogarithms.\footnote{Here in the tabulated entries, the expressions  $ \mathfrak L(x)$, $\frac{1}{\sqrt{3x^2-4}} \mathfrak M(x)$, $ \mathfrak l(x)$,  $\frac{1}{\sqrt{3x^2-4}}\mathfrak m(x)$, $ \tanh^{-1}\frac1x$, and $ \frac{1}{\sqrt{3 x^2-4} }\log \frac{x^2-2+i \sqrt{3 x^2-4}}{x^2-2-i \sqrt{3 x^2-4}}$ can be extended to   holomorphic functions of $x$ in an open set specified by  $ \big|x\big(1-x^{2}\big)\big|>\frac{2}{3\sqrt{3}}$, free from branch cut discontinuities therein. }
\end{proposition}\begin{table}[t]\caption{(Di)logarithmic evaluations of certain infinite series involving $\binom{6k}{3k} $\label{tab:6kLi2}}\begin{scriptsize}\newcounter{AA}\newcounter{BB}\newcounter{CC}\newcounter{DD}\begin{tabular}{c|L|L}\hline\hline \text{\No}&a_k&\displaystyle \sum_{k=1}^\infty \frac{a_{k}\binom{2k}k }{k\binom{3k}k\binom{6k}{3k}}\left[ \frac{2^{2}}{x\big(1-x^{2}\big)} \right]^{2k}\vphantom{\frac{\frac\int1}{\frac1\int}}\\\hline  {\refstepcounter{AA}\AAlabel{eq:sum6}}\text{(A.\theAA{})}&1&\frac{1}{1-3 x^2}\left[6 x\tanh ^{-1}\frac{1}{x}+\frac{i \left(3 x^2-2\right)}{\sqrt{3 x^2-4} }\log \frac{x^2-2+i \sqrt{3 x^2-4}}{x^2-2-i \sqrt{3 x^2-4}}\right]\vphantom{\frac{\frac\int1}{}}\\[8pt] {\refstepcounter{AA}\AAlabel{eq:sum6'}}\text{(A.\theAA{})}&-\mathsf H_k+2\mathsf  H_{2 k}+\mathsf H_{3 k}-2\mathsf  H_{6 k}&\frac{2}{1-3 x^2}\left[x\mathfrak{L}(x)+\frac{i\left(3x^{2}-2\right)\mathfrak{M}(x)}{\sqrt{3 x^2-4}}\right]  \\[8pt] {\refstepcounter{AA}\AAlabel{eq:sum6''}}\text{(A.\theAA{})}&\mathsf H_k-\mathsf H_{2 k}-\frac{1}{2k}+2 \log2&\frac{1 }{1-3 x^2}\left[x \mathfrak{l}(x)-\frac{i\left(3x^{2}-2\right)\mathfrak{m}(x)}{\sqrt{3 x^2-4}} \right]\\[8pt] {\refstepcounter{AA}\AAlabel{eq:sum6'''}}\text{(A.\theAA{})}&\frac1k&-6 \left(\tanh ^{-1}\frac{1}{x}\right)^2-\frac{1}{2} \log ^2\frac{x^2-2+i \sqrt{3 x^2-4}}{x^2-2-i \sqrt{3 x^2-4}}\\[8pt]\hline\hline\multicolumn{3}{c}{}\\\hline\hline \No&a_k&\displaystyle \sum_{k=0}^\infty \frac{a_{k}\binom{2k}k }{(6k+1)\binom{3k}k\binom{6k}{3k}}\left[ \frac{2^{2}}{x\big(1-x^{2}\big)} \right]^{2k}\vphantom{\frac{\frac\int1}{\frac1\int}}\\\hline{\refstepcounter{BB}\BBlabel{eq:(6k+1)}}\text{(B.\theBB{})}&1&\frac{3x \left(1-x^{2}\right)   }{2 \left(1-3 x^2\right)}\left[\tanh ^{-1}\frac{1}{x}-\frac{i  x  }{2 \sqrt{3 x^2-4} }\log \frac{x^2-2+i \sqrt{3 x^2-4}}{x^2-2-i \sqrt{3 x^2-4}}\right]\vphantom{\frac{\frac\int1}{}}\\[8pt]{\refstepcounter{BB}\BBlabel{eq:(6k+1)'}}\text{(B.\theBB{})}&-\mathsf H_k+2\mathsf  H_{2 k}+\mathsf H_{3 k}-2\mathsf  H_{6 k+1}&\frac{x \left(1-x^2\right)}{2 \left(1-3 x^2\right)}\left[ \mathfrak{L}(x)-\frac{3ix\mathfrak{M}(x)}{\sqrt{3 x^2-4}} \right]\\[8pt]{\refstepcounter{BB}\BBlabel{eq:(6k+1)''}}\text{(B.\theBB{})}&\mathsf H_k-\mathsf H_{2 k}+2\log2&\frac{x \left(1-x^2\right)}{4\left(1-3 x^2\right)}\left[ \mathfrak{l}(x)+\frac{3ix\mathfrak{m}(x)}{\sqrt{3 x^2-4}}\right]\\[8pt]\hline\hline
\multicolumn{3}{c}{}\\\hline\hline \No&a_k&\displaystyle \sum_{k=0}^\infty \frac{a_{k}\binom{2k}k }{(6k+5)\binom{3k}k\binom{6k}{3k}}\left[ \frac{2^{2}}{x\big(1-x^{2}\big)} \right]^{2k}\vphantom{\frac{\frac\int1}{\frac1\int}}\\\hline{\refstepcounter{CC}\CClabel{eq:(6k+5)}}\text{(C.\theCC{})}&1&\frac{3 x\left(1-x^2\right)}{2 \left(3 x^2-1\right)}\left[\big(9 x^4-9 x^2+1\big)\tanh ^{-1}\frac{1}{x}+\frac{ i x \left(9 x^4-15 x^2+5\right) }{2 \sqrt{3 x^2-4} }\log\frac{x^2-2+i \sqrt{3 x^2-4}}{x^2-2-i \sqrt{3 x^2-4}}\right]\vphantom{\frac{\frac\int1}{}}\\[8pt]{\refstepcounter{CC}\CClabel{eq:(6k+5)'}}\text{(C.\theCC{})}&\mathsf H_k-\mathsf H_{2 k}+2\log2&\frac{x\left(1-x^2\right) }{4 \left(3 x^2-1\right)}\left[\big(9 x^4-9 x^2+1\big)\mathfrak{l}(x)-\frac{3ix \left(9 x^4-15 x^2+5\right)\mathfrak{m}(x)}{ \sqrt{3 x^2-4}}\right]\\&&{}-\frac{9x \left(1-x^2\right)}{4}\left[\big(3 x^2-2\big)\tanh ^{-1}\frac{1}{x}+\frac{i x \sqrt{3 x^2-4}}{2} \log\frac{x^2-2+i \sqrt{3 x^2-4}}{x^2-2-i \sqrt{3 x^2-4}}\right]\\[8pt]\hline\hline
\multicolumn{3}{c}{}\\\hline\hline \No&a_k&\displaystyle \sum_{k=1}^\infty \frac{a_k\binom{2k}k }{(2k-1)\binom{3k}k\binom{6k}{3k}}\left[ \frac{2^{2}}{x\big(1-x^{2}\big)} \right]^{2k}\vphantom{\frac{\frac\int1}{\frac1\int}}\\\hline{\refstepcounter{DD}\DDlabel{eq:sum6(2k-1)}}\text{(D.\theDD{})}&1&\frac{1}{2 \left(1-x^2\right)\left(1-3 x^2\right)}\left[\frac{3 x^4-3 x^2+2 }{ x }\tanh ^{-1}\frac{1}{x}+\frac{i x^2 \left(3 x^2-5\right) }{2 \sqrt{3 x^2-4} }\log \frac{x^2-2+i \sqrt{3 x^2-4}}{x^2-2-i \sqrt{3 x^2-4}}\right]\vphantom{\frac{\frac\int1}{}}\\[8pt]{\refstepcounter{DD}\DDlabel{eq:sum6(2k-1)'}}\text{(D.\theDD{})}&-\mathsf H_k+2\mathsf  H_{2 k}+\mathsf H_{3 k}-2\mathsf  H_{6 k}-\frac{2}{2 k-1}&\frac{1 }{x \left(1- x^2\right)}\left[\Li_2\left(-\frac{1}{x}\right)-\Li_2\left(\frac{1}{x}\right)\right]+\frac{x}{2 \left(1-x^2\right) \left(1-3 x^2\right)}\left[\big(1+x^2\big)\mathfrak{L}(x)+\frac{ix \left(3 x^2-5\right)\mathfrak{M}(x)}{ \sqrt{3 x^2-4}}\right]  \\[8pt]{\refstepcounter{DD}\DDlabel{eq:sum6(2k-1)''}}\text{(D.\theDD{})}&\mathsf H_k-\mathsf H_{2 k}+\frac{2}{2 k-1}+2 \log2&\frac{1}{2  x\left(1-x^2\right)}\left[ \Li_2\left(\frac{1-x}{1+x}\right)-\Li_2\left(\frac{1+x}{1-x}\right)+\log ^2\left(1+\frac{1}{x}\right)-\log ^2\left(1-\frac{1}{x}\right)\right]\\{}&&{}+\frac{x}{4 \left(1-x^2\right) \left(1-3 x^2\right)}\left[\big(1+x^2\big)\mathfrak{l}(x)-\frac{ ix \left(3 x^2-5\right)\mathfrak{m}(x)}{ \sqrt{3 x^2-4}}\right]\\[8pt]{\refstepcounter{DD}\DDlabel{eq:sum6(2k-1)'''}}\text{(D.\theDD{})}&\frac{1}{2k-1}&-\frac{3 }{2 x \left(1-x^2\right)}\left[\Li_2\left(-\frac{1}{x}\right)-\Li_2\left(\frac{1}{x}\right)\right]+\frac{\mathfrak{L}(x)}{2 x \left(1-x^2\right)}\\&&{}+\frac{1}{4 x \left(1-x^2\right)}\left[\big(3 x^2-2\big)\tanh ^{-1}\frac{1}{x} +\frac{i x \sqrt{3 x^2-4}}{2} \log\frac{x^2-2+i \sqrt{3 x^2-4}}{x^2-2-i \sqrt{3 x^2-4}}\right]\\[8pt]\hline\hline\end{tabular}\end{scriptsize}\end{table}\begin{proof}With Panzer's \texttt{HyperInt} \cite{Panzer2015}, one can verify the following integral identities:{\allowdisplaybreaks\footnotesize\begin{align}
\int_0^1\left( \frac{1}{t+x} -\frac{1}{t-x}\right)\log t\D t={}&\Li_2\left( -\frac{1}{x} \right)-\Li_2\left( \frac{1}{x} \right),\label{eq:logLi2}\\
\sum_{\ell,m\in\{0,1\}}\int_0^1\frac{(-1)^{\ell }}{t-\sigma_{\ell,m}(x)}\log t\D t={}&\sum_{\ell,m\in\{0,1\}}(-1)^{\ell }\Li_2\left( \frac{1}{\sigma_{\ell,m}(x)} \right)=2\left[ \Li_2\left(-\frac{1}{x}\right)-\Li_2\left(\frac{1}{x}\right) \right]-\mathfrak L(x),\\
\sum_{\ell,m\in\{0,1\}}\int_0^1\frac{(-1)^{m }}{t-\sigma_{\ell,m}(x)}\log t\D t={}&\sum_{\ell,m\in\{0,1\}}(-1)^{m }\Li_2\left( \frac{1}{\sigma_{\ell,m}(x)} \right)=\mathfrak M(x),\\\begin{split}
\int_0^1\left( \frac{1}{t+x} -\frac{1}{t-x}\right)\log \frac{1-t^{2}}{t^{2}}\D t={}&\Li_2\left(\frac{1-x}{1+x}\right)-\Li_2\left(\frac{1+x}{1-x}\right)+\log ^2\left(1+\frac{1}{x}\right)-\log ^2\left(1-\frac{1}{x}\right),
\end{split}
\\
\begin{split}
\sum_{\ell,m\in\{0,1\}}\int_0^1\frac{(-1)^{\ell }}{t-\sigma_{\ell,m}(x)}\log \frac{1-t^{2}}{t^{2}}\D t={}&-\sum_{\ell,m\in\{0,1\}}(-1)^{\ell }\left[ \Li_2\left( \frac{1-\sigma_{\ell,m}(x)}{1+\sigma_{\ell,m}(x)}\right) +\log^{2}\left(1+ \frac{1}{\sigma_{\ell,m}(x)} \right)\right]\\={}&2\left[\Li_2\left(\frac{1-x}{1+x}\right)-\Li_2\left(\frac{1+x}{1-x}\right)+\log ^2\left(1+\frac{1}{x}\right)-\log ^2\left(1-\frac{1}{x}\right)\right]-\mathfrak l(x),
\end{split}\\
\sum_{\ell,m\in\{0,1\}}\int_0^1\frac{(-1)^{m }}{t-\sigma_{\ell,m}(x)}\log \frac{1-t^{2}}{t^{2}}\D t={}&-\sum_{\ell,m\in\{0,1\}}(-1)^{m }\left[ \Li_2\left( \frac{1-\sigma_{\ell,m}(x)}{1+\sigma_{\ell,m}(x)}\right) +\log^{2}\left(1+ \frac{1}{\sigma_{\ell,m}(x)} \right)\right]=-\mathfrak m(x),\label{eq:logLi2'}\\\int_{0}^1\frac{1}{1-t^{2}}\log\frac{1-\frac{t^{2}}{w^{2}}}{1-\frac{1}{w^{2}}}\D t={}&\frac{1}{4} \log ^2\frac{w-1}{w+1},\label{eq:log_log}\\\int_{0}^1\frac{1-t^{2}}{t}\log\frac{1-\frac{t}{w}}{1+\frac{t}{w}}\D t={}&\Li_2\left( -\frac{1}{w} \right)-\Li_2\left( \frac{1}{w} \right)+\frac{w^{2}-1}{2}\log \left(\frac{w-1}{w+1}\right)+w,\label{eq:Li2etc}
\end{align}}which will be instrumental in the derivations below.\begin{enumerate}[leftmargin=*,  label=(\Alph*),ref=(\Alph*),
widest=D, align=left] \item
In view of  \eqref{eq:Itab2a}--\eqref{eq:Itab2c}, the infinite  series in (A.\AAref{eq:sum6})--(A.\AAref{eq:sum6''}) can be represented as \begin{align}
\int_0^1\frac{4\left[\frac{t (1-t^2)}{x (1-x^2)}\right]^2}{1-\left[\frac{t (1-t^2)}{x (1-x^2)}\right]^2}\frac{f(t)\D t}{1-t^{2}},
\end{align}where $f(t)=1$, $2\log t$, and $\log\frac{1-t^2}{t^2} $, respectively. The partial fraction\begin{align}
\begin{split}
\frac{\left[\frac{t (1-t^2)}{x (1-x^2)}\right]^2\frac{1}{1-t^{2}}}{1-\left[\frac{t (1-t^2)}{x (1-x^2)}\right]^2}={}&\frac{x }{2 \left(1-3 x^2\right)}\left(\frac{1}{t+x}-\frac{1}{t-x}\right)\\&{}-\frac{1}{4 \left(1-3 x^2\right)}\sum_{\ell,m\in\{0,1\}}\frac{(-1)^{\ell }x+\frac{(-1)^m i(2-3x^{2})}{ \sqrt{3 x^2-4}}}{t-\sigma_{\ell,m}(x)}
\end{split}
\end{align}  and \eqref{eq:logLi2}--\eqref{eq:logLi2'} together bring us the final form of (A.\AAref{eq:sum6})--(A.\AAref{eq:sum6''}).

To verify (A.\AAref{eq:sum6'''}), we represent the corresponding series by\begin{align}
\begin{split}&
\sum_{k=1}^\infty\frac{4}{k}\int_0^1\left[\frac{t\left(1-t^2\right)}{x\left( 1-x^2 \right)}\right]^{2k}\frac{\D t}{1-t^2}=4\int_{0}^{1}\Li_1\left( \left[\frac{t\left(1-t^2\right)}{x\left( 1-x^2 \right)}\right]^{2} \right)\frac{\D t}{1-t^2}\\={}&-4\sum_{w\in S_{x}^+}\int_{0}^{1}\frac{1}{1-t^{2}}\log\frac{1-\frac{t^{2}}{w^{2}}}{1-\frac{1}{w^{2}}}\D t,
\end{split}
\end{align}and appeal to \eqref{eq:log_log}.\item By  \eqref{eq:Itab2a'}--\eqref{eq:Itab2c'}, we may compute the infinite series in (B.\BBref{eq:(6k+1)})--(B.\BBref{eq:(6k+1)''}) through their integral representations\begin{align}
\int_0^1\frac{f(t)\D t}{1-\left[\frac{t (1-t^2)}{x (1-x^2)}\right]^2},
\end{align}where $f(t)=1$, $2\log t$, and $\log\frac{1-t^2}{t^2} $, respectively. Here, the partial fraction
\begin{align}
\frac{1}{1-\left[\frac{t (1-t^2)}{x (1-x^2)}\right]^2}={}&\frac{x \left(1-x^2\right) }{2 \left(1-3 x^2\right)}\left(\frac{1}{t+x}-\frac{1}{t-x}\right)-\frac{x \left(1-x^2\right) }{4 \left(1-3 x^2\right)}\sum_{\ell,m\in\{0,1\}}\frac{(-1)^{\ell }+\frac{3(-1)^m ix}{ \sqrt{3 x^2-4}}}{t-\sigma_{\ell,m}(x)}
\end{align}is helpful.\item Akin to \cite[(3.20)--(3.22)]{SunZhou2024sum3k4k}, we have{\allowdisplaybreaks\begin{align}
\begin{split}&
\left( \frac{1}{6k+1} +\frac{1}{6k+5}\right)\frac{\binom{2k}{k}}{\binom{3k}{k}\binom{6k}{3k}}\\={}&\frac{9}{2^{4k}}\int_0^1\frac{\left[t\left(1-t^2\right)\right]^{2k+2}}{1-t^2}\D t,
\end{split}\\\begin{split}&
\left( \frac{1}{6k+1} +\frac{1}{6k+5}\right)\frac{(-\mathsf H_{k}+2 \mathsf H_{2 k+1}+\mathsf H_{3 k+3}-2\mathsf H_{6 k+6})\binom{2k}{k}}{\binom{3k}{k}\binom{6k}{3k}}\\={}&\frac{18}{2^{4k}}\int_0^1\frac{\left[t\left(1-t^2\right)\right]^{2k+2}\log t}{1-t^2}\D t,
\end{split}\\\begin{split}&\left( \frac{1}{6k+1} +\frac{1}{6k+5}\right)\frac{(\mathsf H_{2 k+1}+\mathsf H_{3 k+3}-2\mathsf H_{6 k+6}+2 \log 2)\binom{2k}{k}}{\binom{3k}{k}\binom{6k}{3k}}\\
={}&\frac{9}{2^{4k}}\int_0^1\frac{\left[t\left(1-t^2\right)\right]^{2k+2}\log \left(1-t^2\right)}{1-t^2}\D t.
\end{split}
\end{align}}     Like \cite[(3.24)]{SunZhou2024sum3k4k}, we have\begin{align}
\begin{split}&
\sum_{k=0}^\infty\left( \frac{1}{6k+1} +\frac{1}{6k+5}\right)\frac{( \mathsf H_{2 k+1}-\mathsf H_{2k})\binom{2k}{k}}{\binom{3k}{k}\binom{6k}{3k}}\left[ \frac{2^{2}}{x\big(1-x^{2}\big)} \right]^{2k}\\={}&\frac{3}{2}\sum_{k=0}^\infty\left( \frac{1}{6k+1} -\frac{1}{6k+5}\right)\frac{\binom{2k}{k}}{\binom{3k}{k}\binom{6k}{3k}}\left[ \frac{2^{2}}{x\big(1-x^{2}\big)} \right]^{2k}\\={}&3\int_0^1\frac{\D t}{1-\left[\frac{t (1-t^2)}{x (1-x^2)}\right]^2}-\frac{27}{2}\int_0^1\frac{t^{2}\left(1-t^2\right)\D t}{1-\left[\frac{t (1-t^2)}{x (1-x^2)}\right]^2}.
\end{split}
\end{align}Therefore, both (C.\CCref{eq:(6k+5)}) and (C.\CCref{eq:(6k+5)'}) are analogs of the cases treated in panel (B).\item The series identities in   (D.\DDref{eq:sum6(2k-1)})--(D.\DDref{eq:sum6(2k-1)''}) build on \eqref{eq:2k-1_int_repn1}--\eqref{eq:2k-1_int_repn3}. We omit the routine computations involving partial fractions and invocations of \eqref{eq:logLi2}--\eqref{eq:logLi2'}.

To evaluate the infinite series in (D.\DDref{eq:sum6(2k-1)'''}), we compare its integral representation\begin{align}
\begin{split}&
\sum_{k=1}^\infty\frac{1}{2k-1}\int_0^1\left[\frac{t\left(1-t^2\right)}{x\left( 1-x^2 \right)}\right]^{2k}\frac{\D t}{t^2}=-\frac{1}{2x\left( 1-x^2 \right)}\int_0^1\frac{1-t^{2}}{t}\log\frac{1-\frac{t (1-t^2)}{x (1-x^2)}}{1+\frac{t (1-t^2)}{x (1-x^2)}}\D t\\={}&-\frac{1}{2x\left( 1-x^2 \right)}\sum_{w\in S_x^+}\int_0^1\frac{1-t^{2}}{t}\log\frac{1- \frac{t}{w}}{1+\frac{t}{w}}\D t
\end{split}
\end{align}to \eqref{eq:Li2etc}.\qedhere\end{enumerate}

\end{proof}
\begin{corollary}[cf.\ {\cite[Conjecture 5.5(ii)]{Sun2023}}]\label{cor:Z2(8)}
We have {\allowdisplaybreaks\begin{align}\begin{split}&
\sum_{k=1}^\infty\frac{\binom{2k}k2^{3k}}{(2k-1)\binom{3k}k\binom{6k}{3k}}\\={}&\frac{2 \sqrt{2}\widetilde \lambda}{5}-\frac{\pi }{10 \sqrt{2}},
\end{split}\label{eq:T60spec1}\\\begin{split}&
\sum_{k=1}^\infty\frac{\left[ -\mathsf H_k+2\mathsf  H_{2 k}+\mathsf H_{3 k}-2\mathsf  H_{6 k}-\frac{2}{2 k-1} \right]\binom{2k}k2^{3k}}{(2k-1)\binom{3k}k\binom{6k}{3k}}\\={}&\frac{2\sqrt{2}\left[ \RE\Li_{1,1}(i,\theta)+\I\Li_2(\theta) -\frac{G}{4}\right]}{5}-\frac{3 \lambda ^2}{40 \sqrt{2}}+\frac{\lambda  \widetilde{\lambda }}{10 \sqrt{2}}-\frac{\widetilde{\lambda }^2}{10 \sqrt{2}}-\frac{83 \pi ^2}{480 \sqrt{2}},
\end{split}\label{eq:T60spec2}\\\begin{split}&
\sum_{k=1}^\infty\frac{\left[\mathsf H_k-\mathsf H_{2 k}+\frac{2}{2 k-1}+2 \log2 \right]\binom{2k}k2^{3k}}{(2k-1)\binom{3k}k\binom{6k}{3k}}\\={}&-\frac{\sqrt{2}\left[ 15\RE\Li_{1,1}(i,\theta)+\I\Li_2(\theta) -\frac{G}{4}\right]}{5}+\frac{9 \lambda ^2}{16 \sqrt{2}}+\frac{17 \lambda  \widetilde{\lambda }}{20 \sqrt{2}}+\frac{3 \widetilde{\lambda }^2}{4 \sqrt{2}}-\frac{\pi  \lambda }{20 \sqrt{2}}-\frac{\pi ^2}{320 \sqrt{2}},
\end{split}\label{eq:T60spec3}\\\begin{split}&
\sum_{k=1}^\infty\frac{\binom{2k}k2^{3k}}{(2k-1)^{2}\binom{3k}k\binom{6k}{3k}}\\={}&\sqrt{2}\RE\Li_{1,1}(i,\theta)-\frac{3 \lambda ^2}{16 \sqrt{2}}+\frac{\lambda  \widetilde{\lambda }}{4 \sqrt{2}}-\frac{\widetilde{\lambda }}{\sqrt{2}}-\frac{\widetilde{\lambda }^2}{4 \sqrt{2}}+\frac{\pi }{4 \sqrt{2}}+\frac{13 \pi ^2}{192 \sqrt{2}},
\end{split}\label{eq:T60spec4}\end{align}}where $ \theta\colonequals e^{\pi i/4}$, $\lambda\colonequals\log2 $, $\widetilde \lambda\colonequals\log\big(1+\sqrt{2}\big)$, and $ G\colonequals\I\Li_2(i)$. Consequently, we have \eqref{eq:conj5.5ii1}--\eqref{eq:conj5.5ii3}.\end{corollary}\begin{proof}Using Au's \texttt{MultipleZetaValues} package \cite{Au2022a}, one can check that  {\allowdisplaybreaks\begin{align}
\begin{split}
&\Li_2\left( -\frac{1}{\sqrt{2}} \right)-\Li_2\left( \frac{1}{\sqrt{2}} \right)\\={}&4\RE\Li_{1,1}(i,\theta)-\frac{3 \lambda ^2}{8}+\frac{\lambda  \widetilde{\lambda }}{2}-\frac{\widetilde{\lambda }^2}{2}-\frac{11 \pi ^2}{96}\in\mathfrak Z_2(8),
\end{split}\\\begin{split}
&\Li_2\left(\frac{1-\sqrt{2}}{1+\sqrt{2}}\right)-\Li_2\left(\frac{1+\sqrt{2}}{1-\sqrt{2}}\right)+\log ^2\left(1+\frac{1}{\sqrt{2}}\right)-\log ^2\left(1-\frac{1}{\sqrt{2}}\right)\\={}&-24\RE\Li_{1,1}(i,\theta)+\frac{9 \lambda ^2}{4}+\lambda  \widetilde{\lambda }+3 \widetilde{\lambda }^2-\frac{5 \pi ^2}{16}\in\mathfrak Z_2(8),
\end{split}
\end{align}and \begin{align}
\mathfrak{L}\big(\sqrt{2}\big)={}&8\RE\Li_{1,1}(i,\theta)-\frac{3 \lambda ^2}{4}+\lambda  \widetilde{\lambda }-\widetilde{\lambda }^2-\frac{23 \pi ^2}{48}\in\mathfrak Z_2(8),\\\mathfrak{M}\big(\sqrt{2}\big)={}&-4i\I\Li_2(\theta)+iG\in\mathfrak Z_2(8),\\
\mathfrak{l}\big(\sqrt{2}\big)={}&-60\RE\Li_{1,1}(i,\theta)+\frac{45 \lambda ^2}{8}+\frac{9 \lambda  \widetilde{\lambda }}{2}+\frac{15 \widetilde{\lambda }^2}{2}-\frac{17 \pi ^2}{32}\in\mathfrak Z_2(8),\\\mathfrak{m}\big(\sqrt{2}\big)={}&-4i\I\Li_2(\theta)+i G-\frac{i \pi  \lambda }{2}\in\mathfrak Z_2(8).
\end{align}}Here, we have $\theta\colonequals e^{\pi i/4} $, $\lambda\colonequals \log2 $,  $\widetilde \lambda \colonequals \log\big(1+\sqrt2\big)$, and $G\colonequals \I\Li_2(i) $.

Drawing on the information from the last paragraph and panel (D) of Table \ref{tab:6kLi2}, we can verify \eqref{eq:T60spec1}--\eqref{eq:T60spec4}.

One can produce analogs of \eqref{eq:T60spec1}--\eqref{eq:T60spec4} where $k$ takes the place of  $2k-1 $ in the denominators of the summands, by either exploiting panel (A) of Table \ref{tab:6kLi2} or referring back to Table \ref{tab:Zk(8)6k}.

 With the aforementioned data in hand, one can verify \eqref{eq:conj5.5ii1}--\eqref{eq:conj5.5ii3}.
\end{proof}
\subsection{Trilogarithmic evaluations}
In addition to proving \cite[Conjecture 5.5(iii)]{Sun2023} in this subsection, we will also encounter various infinite series resembling \eqref{eq:conj5.5iii}, which can all be expressed in terms of $ \mathfrak Z_r(8),r\in\{1,2,3\}$.

In the next  proposition, we extend the evaluations in Table \ref{tab:6kLi2} a little further, relying on the integral representations in \eqref{eq:Itab2a}--\eqref{eq:Itab2c'}, as well as the following generating functions for $ r\in\{1,2\}$ and $ |z|<1$ \cite[\S5.1]{Mezo2014}:\begin{align}
\sum_{k=1}^\infty\mathsf H_k^{(r)}z^k={}&\frac{\Li_r(z)}{1-z},&\sum_{k=1}^\infty\mathsf H_{2k-1}^{(r)}z^{2k}=\frac{z}{2}{}&\left[\frac{\Li_r(z)}{1-z}-\frac{\Li_r(-z)}{1+z}\right].\label{eq:MezoG}
\end{align}

\begin{proposition}For   $ \big|x\big(1-x^{2}\big)\big|>\frac{2}{3\sqrt{3}}$, we have the  evaluations in Table \ref{tab:6kLi3} involving the functions $ \mathfrak l(x)$ and $\mathfrak m(x)
 $ [defined in \eqref{eq:l(x)defn} and \eqref{eq:m(x)defn}], along with  the following identities  \cite{Panzer2015}\begin{align}
\begin{split}
A_{3}(w)\colonequals {}&\int_0^1\frac{\log^2\frac{t^2}{1-t^2}}{t-w}\D t\\={}&2\left[G(w,-1,-1;1)+G(w,-1,1;1)+G(w,1,-1;1) -3G(w,1,1;1)\right]\\{}&+4\left[G(0,w,-1;1)-G(0,w,1;1)\right]+\frac{4}{3}\log^3\left( 1-\frac{1}{w} \right)+\frac{4\pi^{2}}{3}\log\left( 1-\frac{1}{w} \right),\label{eq:A3defn}
\end{split}\\\begin{split}\mathscr A_3(a,b)\colonequals {}&-\int_0^1\frac{\log\left( 1-\frac{t}{a} \right)}{t-b}\log\frac{t^2}{1-t^2}\D t
\\={}&2G(0,b,a;1)+\sum_{\tau\in\{-1,1\}}\left[G(b,a,\tau;1)+G(b,\tau,a;1)\right].
\end{split}\label{eq:AA3_defn}
\end{align}
Throughout Table \ref{tab:6kLi3},  it is understood that Panzer's logarithmic regularization \cite[\S2.3]{Panzer2015} replaces each divergent  GPL $G(1,a,b;1) $ by a convergent expression $ G(0,1-b,1-a;1)+G(1-b,0,1-a;1)$.\footnote{Note that  the  following {\texttt{Maple}} code \begin{quote}{\texttt{fibrationBasis(eval(regHlog(Hlog(1, [1, a, b])), Hlog(1, [1]) = 0) - (Hlog(1, [0, 1 - b, 1 - a]) + Hlog(1, [1 - b, 0, 1 - a])), [a, b]);}}\end{quote}returns a vanishing output.}\end{proposition}\begin{table}[hp]\caption{Evaluations of certain infinite series involving $\binom{6k}{3k} $ by GPLs of weights up to   $3$ \label{tab:6kLi3}}\begin{tiny}\newcounter{AAA}\newcounter{BBB}\begin{tabular}{c|L|L}\hline\hline \text{\No}&a_k&\displaystyle \sum_{k=1}^\infty \frac{a_{k}\binom{2k}k }{k\binom{3k}k\binom{6k}{3k}}\left[ \frac{2^{2}}{x\big(1-x^{2}\big)} \right]^{2k}\vphantom{\frac{\frac\int1}{\frac1\int}}\\\hline  {\refstepcounter{AAA}\AAAlabel{eq:Li3sum6k1}}\text{(A.\theAAA{})}&\frac{1}{k^{2}}&\displaystyle4\sum_{\tau\in\{-1,1\}}\sum _{\alpha\in\{-1,0,1\}}\sum _{w\in S_{x}}G(\tau,\alpha ,w;1)\tau\vphantom{\frac{1}{}}\\[15pt] {\refstepcounter{AAA}\AAAlabel{eq:Li3sum6k2}}\text{(A.\theAAA{})}&\frac{-\mathsf H_k+2\mathsf  H_{2 k}+\mathsf H_{3 k}-2\mathsf  H_{6 k}}{k}&\displaystyle-4\sum_{\tau\in\{-1,1\}}\sum_{w\in S_{x}}G(0,\tau,w;1)\tau  \\[15pt] {\refstepcounter{AAA}\AAAlabel{eq:Li3sum6k3}}\text{(A.\theAAA{})}&\frac{\mathsf H_{2 k-1}+\mathsf H_{3 k}-2\mathsf H_{6 k}+ \log2}{k}&\displaystyle-2\sum_{\tau,\tau'\in\{-1,1\}}\sum_{w\in S_{x}}G(\tau',\tau,w;1)\tau\\[15pt] {\refstepcounter{AAA}\AAAlabel{eq:Li3sum6k4}}\text{(A.\theAAA{})}&\frac{\mathsf H_{k-1}}{k}&\displaystyle -2\sum_{\tau\in\{-1,1\}}\sum_{w,w'\in S_x}G(\tau ,w,w';1)\tau\\[15pt] {\refstepcounter{AAA}\AAAlabel{eq:Li3sum6k5}}\text{(A.\theAAA{})}&\mathsf H_{k-1}^{(2)}&\displaystyle \tfrac{2}{1-3x^{2}}\sum_{\ell,m\in\{0,1\}}\sum _{\alpha\in\{-1,0,1\}}\sum _{w\in S_{x}}\left[(-1)^{\ell}x-\tfrac{(-1)^{m}i\big( 3x^2 -2\big)}{\sqrt{3x^{2}-4}} \right]G(\sigma_{\ell,m}(x),\alpha ,w;1)\\[12pt]&&\displaystyle{}+\tfrac{4x}{1-3x^{2}}\sum _{\alpha\in\{-1,0,1\}}\sum _{w\in S_{x}}[G(x,\alpha ,w;1)-G(-x,\alpha ,w;1)]\\[15pt] {\refstepcounter{AAA}\AAAlabel{eq:Li3sum6k6}}\text{(A.\theAAA{})}&\mathsf H_{2k-1}^{(2)}&\displaystyle\tfrac{1}{1-3x^{2}}\sum_{m\in\{0,1\}}\sum _{\alpha\in\{-1,0,1\}}\sum _{w\in S^-_{x}}\left[x-\tfrac{(-1)^{m}i\big( 3x^2 -2\big)}{\sqrt{3x^{2}-4}} \right]G(\sigma_{0,m}(x),\alpha,w;1)\\[12pt]&&\displaystyle{}+\tfrac{1}{1-3x^{2}}\sum_{m\in\{0,1\}}\sum _{\alpha\in\{-1,0,1\}}\sum _{w\in S^+_{x}}\left[-x-\tfrac{(-1)^{m}i\big( 3x^2 -2\big)}{\sqrt{3x^{2}-4}} \right]G(\sigma_{1,m}(x),\alpha,w;1)\\[12pt]&&\displaystyle{}-\tfrac{2 x}{1-3x^2 }\sum _{\alpha\in\{-1,0,1\}}\left[\sum _{w\in S^-_{x}}G(-x,\alpha,w;1)-\sum _{w\in S^+_{x}}G(x,\alpha,w;1)\right]\\[15pt]\hline\hline\multicolumn{3}{c}{}\\\hline\hline \text{\No}&a_k&\displaystyle \sum_{k=1}^\infty \frac{ a_{k}\binom{2k}k}{(2k-1)\binom{3k}k\binom{6k}{3k}}\left[ \frac{2^{2}}{x\big(1-x^{2}\big)} \right]^{2k}\vphantom{\frac{\frac\int1}{\frac1\int}}\\\hline  {\refstepcounter{BBB}\BBBlabel{eq:Li3sum6k1'}}\text{(B.\theBBB{})}&\frac{1}{(2k-1)^2}&\displaystyle -\tfrac{1}{2x\left( 1-x^2 \right)}\left(\sum _{w\in S^+_{x}}-\sum _{w\in S^-_{x}} \right)\left[\sum _{\alpha\in\{-1,0,1\}}G(0,\alpha ,w;1)- \tfrac{G(0,w;1)}{2}-\tfrac{3w^{2}}{4}\log\left( 1-\tfrac{1}{w} \right)\right]\vphantom{\frac{\frac\int\int}{}}
\\[15pt]{\refstepcounter{BBB}\BBBlabel{eq:Li3sum6k2'}}\text{(B.\theBBB{})}&\frac{-\mathsf H_k+2\mathsf  H_{2 k}+\mathsf H_{3 k}-2\mathsf  H_{6 k}-\frac{2}{2 k-1}}{2k-1}&\displaystyle -\tfrac{1}{x \left(1-x^2\right)}\left(\sum _{w\in S^+_{x}}-\sum _{w\in S^-_{x}} \right)\left[G(w,1,1;1)+G(0,w,1;1)+\tfrac{w^2G(w,1;1)}{2}-\tfrac{1}{6}\log^3\left( 1-\tfrac{1}{w} \right)-\tfrac{w^{2}}{4}\log^2\left( 1-\tfrac{1}{w} \right) \right.\\&&{}\displaystyle\left.-\tfrac{w^{2}}{4}\log\left( 1-\tfrac{1}{w} \right)\right]
\\[5pt]{\refstepcounter{BBB}\BBBlabel{eq:Li3sum6k3'}}\text{(B.\theBBB{})}&\frac{\mathsf H_{2 k}+\mathsf H_{3 k}-2\mathsf  H_{6 k}+2 \log 2}{2k-1}&\displaystyle-\tfrac{1}{2x \left(1-x^2\right)}\left(\sum _{w\in S^+_{x}}-\sum _{w\in S^-_{x}} \right)\Bigg\{\sum_{\tau\in\{-1,1\}}\left[ G(0,\tau,w;1)+G(0,w,\tau;1)-\tfrac{1-w^2}{2} G(w,\tau;1)\right]-\tfrac{w^2}2\log\left( 1-\tfrac{1}{w} \right)\Bigg\}
\\[15pt]{\refstepcounter{BBB}\BBBlabel{eq:Li3sum6k4'}}\text{(B.\theBBB{})}&\frac{\mathsf H_{2k-2}}{2k-1}&\displaystyle\tfrac{1}{2x\left( 1-x^2 \right)}\left(\sum _{w,w'\in S^+_{x}}-\sum _{w,w'\in S^-_{x}} \right)\left[ G(0,w,w';1)-\tfrac{1-w^{2}}{2} G(w,w';1)\right]-\tfrac{3}{8x\left( 1-x^2 \right)}\left(\sum _{w\in S^+_{x}}-\sum _{w\in S^-_{x}} \right) w^2\log\left( 1-\tfrac{1}{w} \right)
\\[15pt]{\refstepcounter{BBB}\BBBlabel{eq:Li3sum6k5'}}\text{(B.\theBBB{})}&\mathsf H_{k-1}^{(2)}&\displaystyle\tfrac{x}{2\left(1-x^2\right)\left(1-3x^{2}\right)}\sum_{\ell,m\in\{0,1\}}\sum _{\alpha\in\{-1,0,1\}}\sum _{w\in S_{x}}\left[(-1)^{\ell}\big(1+x^{2}\big)-\tfrac{(-1)^{m}ix\big( 3x^2 -5\big)}{\sqrt{3x^{2}-4}} \right]G(\sigma_{\ell,m}(x),\alpha ,w;1)\\[12pt]&&\displaystyle{}+\tfrac{1-x^2}{x\left(1-3x^{2}\right)}\sum _{\alpha\in\{-1,0,1\}}\sum _{w\in S_{x}}[G(x,\alpha ,w;1)-G(-x,\alpha ,w;1)]
\\[15pt]{\refstepcounter{BBB}\BBBlabel{eq:Li3sum6k6'}}\text{(B.\theBBB{})}&\mathsf H_{2k-1}^{(2)}&\displaystyle\tfrac{x}{4\left(1-x^2\right)\left(1-3x^{2}\right)}\sum_{m\in\{0,1\}}\sum _{\alpha\in\{-1,0,1\}}\sum _{w\in S^-_{x}}\left[\big(1+x^{2}\big)-\tfrac{(-1)^{m}ix\big( 3x^2 -5\big)}{\sqrt{3x^{2}-4}} \right]G(\sigma_{0,m}(x),\alpha,w;1)\\&&\displaystyle {}+\tfrac{x}{4\left(1-x^2\right)\left(1-3x^{2}\right)}\sum_{m\in\{0,1\}}\sum _{\alpha\in\{-1,0,1\}}\sum _{w\in S^+_{x}}\left[-\big(1+x^{2}\big)-\tfrac{(-1)^{m}ix\big( 3x^2 -5\big)}{\sqrt{3x^{2}-4}} \right]G(\sigma_{1,m}(x),\alpha,w;1)\\&&\displaystyle {}-\tfrac{1-x^2}{2x\left(1-3x^{2}\right)}\sum _{\alpha\in\{-1,0,1\}}\left[\sum _{w\in S^-_{x}}G(-x,\alpha,w;1)-\sum _{w\in S^+_{x}}G(x,\alpha,w;1)\right]\\&&{}\displaystyle-\tfrac{1}{2x\left(1-x^2\right)}\sum _{\alpha\in\{-1,0,1\}}\left(\sum _{w\in S^+_{x}}-\sum _{w\in S^-_{x}} \right)G(0,\alpha,w;1)\\[15pt]
\hline\hline\end{tabular}\end{tiny}\end{table}\setcounter{table}{3}\begin{table}[t]\caption{ \textit{(Continued)}}\begin{tiny}\begin{tabular}{c|L|L}\hline\hline \text{\No}&a_k&\displaystyle \sum_{k=1}^\infty \frac{ a_{k}\binom{2k}k}{k\binom{3k}k\binom{6k}{3k}}\left[ \frac{2^{2}}{x\big(1-x^{2}\big)} \right]^{2k}\vphantom{\frac{\frac\int1}{\frac1\int}}\\\hline  {\refstepcounter{AAA}\AAAlabel{eq:Li3sum6k5+6}}\text{(A.\theAAA{})}&\mathsf  H_{k-1}^{(2)}-5\mathsf H_{2k-1}^{(2)}&\displaystyle-\sum_{\ell,m\in\{0,1\}}\left[(-1)^{\ell}x-\tfrac{(-1)^{m}i\big( 3x^2 -2\big)}{\sqrt{3x^{2}-4}} \right]\tfrac{A_3(\sigma_{\ell,m}(x))}{1-3x^{2}}-\tfrac{2x[A_{3}(x)-A_{3}(-x)]}{1-3x^{2}}\vphantom{\frac{\frac\int1}{}}{}\\[10pt]&&{}\displaystyle-\sum_{w\in S^+_x}\left\{\sum_{\tau\in\{-1,1\}}\mathscr A_3(w,\tau)\tau -\sum_{m\in\{0,1\}}\left[x-\tfrac{(-1)^{m}i\big( 3x^2 -2\big)}{\sqrt{3x^{2}-4}} \right]\tfrac{\mathscr A_3(w,\sigma_{0,m}(x))}{3x^2-1}+\tfrac{2x\mathscr A_3(w,-x)}{3x^{2}-1}\right\}
\\[12pt]{}&&{}\displaystyle-\sum_{w\in S^-_x}\left\{\sum_{\tau\in\{-1,1\}}\mathscr A_3(w,\tau)\tau -\sum_{m\in\{0,1\}}\left[-x-\tfrac{(-1)^{m}i\big( 3x^2 -2\big)}{\sqrt{3x^{2}-4}} \right]\tfrac{\mathscr A_3(w,\sigma_{1,m}(x))}{3x^2-1}-\tfrac{2x\mathscr A_3(w,x)}{3x^{2}-1}\right\}\\[12pt]&&{}\displaystyle -\sum_{\tau,\tau'\in\{-1,1\}}\sum_{w\in S_{x}}G(\tau',\tau,w;1)\tau+2\sum_{\tau\in\{-1,1\}}\sum_{w\in S_{x}}G(0,\tau,w;1)\tau\\[12pt]&&{}\displaystyle -\left\{\tfrac{2}{1-3 x^2}\left[ x \mathfrak{l}(x)-\tfrac{i\big( 3x^2 -2\big) \mathfrak{m}(x)}{\sqrt{3 x^2-4}} \right]+\left[3\left(  \tanh ^{-1}\tfrac{1}{x} \right)^{2}
+\tfrac{1}{4}\log^{2} \tfrac{x^2-2+i \sqrt{3 x^2-4}}{x^2-2-i \sqrt{3 x^2-4}}\right]\right\}\log2\\[12pt]{}&&{}\displaystyle-\tfrac{2 \pi ^2}{3\left( 1-3 x^2 \right)}\left[ 6 x \tanh ^{-1}\tfrac{1}{x}+\tfrac{i\big( 3x^2 -2\big)
}{\sqrt{3 x^2-4}} \log \tfrac{x^2-2+i \sqrt{3 x^2-4}}{x^2-2-i \sqrt{3 x^2-4}}\right]\\[10pt]\hline\hline\multicolumn{3}{c}{}\\\hline\hline \No&a_k&\displaystyle \sum_{k=1}^\infty \frac{ a_{k}\binom{2k}k }{(2k-1)\binom{3k}k\binom{6k}{3k}}\left[ \frac{2^{2}}{x\big(1-x^{2}\big)} \right]^{2k}\vphantom{\frac{\frac\int1}{\frac1\int}}\\\hline  {\refstepcounter{BBB}\BBBlabel{eq:Li3sum6k5'+6'}}\text{(B.\theBBB{})}&\mathsf  H_{k}^{(2)}-5\mathsf H_{2k}^{(2)}&\displaystyle-\sum_{\ell,m\in\{0,1\}}\left[(-1)^{\ell}\big(1+x^{2}\big)-\tfrac{(-1)^{m}ix\big( 3x^2 -5\big)}{\sqrt{3x^{2}-4}}
\right]\tfrac{xA_3(\sigma_{\ell,m}(x))}{24\left(1-x^2\right)\left(1-3x^{2}\right)}-\tfrac{\big(1-x^2\big)[A_{3}(x)-A_{3}(-x)]}{12x\left(1-3x^{2}\right)}\vphantom{\frac{\frac\int1}{}}{}\\{}&&\displaystyle{}-\sum_{w\in S^+_x}\left\{\tfrac{\mathscr A_3(w,0)}{2 x \left(1-x^2\right)} +\sum_{m\in\{0,1\}}\left[\big(1+x^{2}\big)-\tfrac{(-1)^{m}ix\big( 3x^2 -5\big)}{\sqrt{3x^{2}-4}} \right]\tfrac{x\mathscr A_3(w,\sigma_{0,m}(x))}{4 \left(1-x^2\right) \left(1-3 x^2\right)}-\tfrac{\big(1-x^{2}\big)\mathscr A_3(w,-x)}{2 x \left(1-3 x^2\right)}\right\}\\{}&&\displaystyle{}-\sum_{w\in S^-_x}\left\{-\tfrac{\mathscr A_3(w,0)}{2 x \left(1-x^2\right)} +\sum_{m\in\{0,1\}}\left[-\big(1+x^{2}\big)-\tfrac{(-1)^{m}ix\big( 3x^2 -5\big)}{\sqrt{3x^{2}-4}} \right]\tfrac{x\mathscr A_3(w,\sigma_{1,m}(x))}{4 \left(1-x^2\right) \left(1-3 x^2\right)}+\tfrac{\big(1-x^{2}\big)\mathscr A_3(w,x)}{2 x \left(1-3 x^2\right)}\right\}\\&&\displaystyle{}+\tfrac{1}{x \left(1-x^2\right)}\left(\sum _{w\in S^+_{x}}-\sum _{w\in S^-_{x}} \right)\sum_{\tau\in\{-1,1\}}\left[ G(0,\tau,w;1)+G(0,w,\tau;1)-\tfrac{1-w^2}{2} G(w,\tau;1)\right]\\&&\displaystyle{}-\tfrac{2}{x \left(1-x^2\right)}\left(\sum _{w\in S^+_{x}}-\sum _{w\in S^-_{x}} \right)\left[G(w,1,1;1)+G(0,w,1;1)+\tfrac{w^2G(w,1;1)}{2}-\tfrac{1}{6}\log^3\left( 1-\tfrac{1}{w} \right)-\tfrac{w^{2}}{4}\log^2\left( 1-\tfrac{1}{w} \right) \right]\\&&{}-\frac{\frac{1}{2}+\log2}{
x\left(1-x^2\right)}\left[ \Li_2\left(\frac{1-x}{1+x}\right)-\Li_2\left(\frac{1+x}{1-x}\right)+\log ^2\left(1+\frac{1}{x}\right)-\log ^2\left(1-\frac{1}{x}\right)\right]-\frac{\left(\frac{1}{2}+\log2\right)x}{2\left(1-x^2\right) \left(1-3 x^2\right)}\left[\big(1+x^2\big)\mathfrak{l}(x)-\frac{ ix  \big(3 x^2-5\big)\mathfrak{m}(x)}{ \sqrt{3 x^2-4}}\right]\\&&\displaystyle{}+\tfrac{2}{x\left(1-x^2\right)}\left(\sum _{w\in S^+_{x}}-\sum _{w\in S^-_{x}} \right)\left[\sum _{\alpha\in\{-1,0,1\}}G(0,\alpha ,w;1)- \tfrac{G(0,w;1)}{2}-\tfrac{3w^{2}}{4}\log\left( 1-\tfrac{1}{w} \right)\right]\\&&\displaystyle{}-\tfrac{ \pi ^2-3}{3 \left(1-x^2\right) \left(1-3 x^2\right)}\left[\tfrac{3 x^4-3 x^2+2 }{ x }\tanh ^{-1}\tfrac{1}{x}+\tfrac{i x^2 \big(3 x^2-5\big) }{2 \sqrt{3 x^2-4} }\log \tfrac{x^2-2+i \sqrt{3 x^2-4}}{x^2-2-i \sqrt{3 x^2-4}}\right]-\tfrac{3}{2}\left(  \tanh ^{-1}\tfrac{1}{x} \right)^{2}
-\tfrac{1}{8}\log^{2} \tfrac{x^2-2+i \sqrt{3 x^2-4}}{x^2-2-i \sqrt{3 x^2-4}}\\&&\displaystyle{}-\tfrac{1}{1-3 x^2 }\left[ 6 x \tanh ^{-1}\tfrac{1}{x}+\tfrac{i\big( 3x^2 -2\big)
}{\sqrt{3 x^2-4}} \log \tfrac{x^2-2+i \sqrt{3 x^2-4}}{x^2-2-i \sqrt{3 x^2-4}}\right]+\tfrac{1}{2 \left(1-3 x^2\right)}\left[ x \mathfrak{l}(x)-\tfrac{i \big( 3x^2 -2\big) \mathfrak{m}(x)}{\sqrt{3 x^2-4}} \right]\\[8pt]\hline\hline\end{tabular}\end{tiny}\end{table}
\begin{proof}\begin{enumerate}[leftmargin=*,  label=(\Alph*),ref=(\Alph*),
widest=B, align=left] \item In view of \eqref{eq:Itab2a}, we can represent the series in  (A.\AAAref{eq:Li3sum6k1}) by\begin{align}
2\int_{0}^1\left( \frac{1}{t+1}-\frac{1}{t-1} \right)\Li_2\left(\left[ \frac{t\big(1-t^{2}\big)}{x\big(1-x^{2}\big)} \right]^{2} \right)\D t.
\end{align}From \eqref{eq:LiSx}, we know that \begin{align}\Li_2
\left(\left[ \frac{t\big(1-t^{2}\big)}{x\big(1-x^{2}\big)} \right]^{2} \right)=-2\sum _{\alpha\in\{-1,0,1\}}\sum _{w\in S_{x}}G(\alpha ,w;t),\label{eq:Li2fib6k}
\end{align}and the last column of    (A.\AAAref{eq:Li3sum6k1})  immediately follows from the GPL recursion in \eqref{eq:GPL_rec}.

By Panzer's  \texttt{HyperInt} \cite{Panzer2015}, one can check that {\allowdisplaybreaks
\begin{align}
\int_0^1\frac{\log\left( 1-\frac{t}{w} \right)}{1-t^{2}}\log t\D t={}&\frac{G(0,1,w;1)-G(0,-1,w;1)}{2}\label{eq:loglogA}\end{align}and\begin{align}\small\begin{split}
\int_0^1\frac{\log\frac{1-\frac{t}{w}}{1-\frac{1}{w}}}{1-t^{2}}\log \frac{1-t^{2}}{2}\D t={}&\frac{G(-1,1,w;1)-G(-1,-1,w;1)}{2}+\underset{\text{regularized }G(1,1,w;1)/2}{\underbrace{\frac{G(0,1-w,0;1)+G(1-w,0,0;1)}{2}}
}\\{}&-\underset{\text{regularized }G(1,-1,w;1)/2}{\underbrace{\frac{G(0,1-w,2;1)+G(1-w,0,2;1)}{2}}
}-\frac{1}{4}\left(\log^22- \frac{\pi^2}{3} \right)\log\left( 1-\frac{1}{w} \right),\end{split}\label{eq:loglogB}
\end{align}
}so both (A.\AAAref{eq:Li3sum6k2}) and (A.\AAAref{eq:Li3sum6k3}) hold true.

The infinite series in (A.\AAAref{eq:Li3sum6k4}) has an integral representation \begin{align}
2\int_0^1\log^{2}\left(1-\left[ \frac{t\big(1-t^{2}\big)}{x\big(1-x^{2}\big)} \right]^{2} \right)\frac{\D t}{1-t^2},
\end{align} because (cf.\ \cite[(25)]{Mezo2014})\begin{align}
\sum_{k=1}^\infty\frac{\mathsf H_{k-1}}{k}z^k=\int_{0}^z\sum_{k=1}^\infty\mathsf H_{k-1}x^{k-1}\D x=-\int_{0}^z\frac{\log(1-x)}{1-x}\D x=\frac{\log^{2}(1-z)}{2}\label{eq:Hkloglog}
\end{align}for $ |z|<1$. Here, by the shuffle algebra of GPLs \cite[(2.4)]{Frellesvig2016}, we have\begin{align}
\frac{1}{2}\log^{2}\left(1-\left[ \frac{t\big(1-t^{2}\big)}{x\big(1-x^{2}\big)} \right]^{2} \right)=\sum_{w,w'\in S_x}G(w,w';t),
\end{align}so the last column of (A.\AAAref{eq:Li3sum6k4})  is a consequence of the GPL recursion \eqref{eq:GPL_rec}.

The series in   (A.\AAAref{eq:Li3sum6k5}) is equal to the following integral:\begin{align}
4\int_0^1\frac{\left[ \frac{t(1-t^{2})}{x(1-x^{2})} \right]^{2}\Li_2\left(\left[ \frac{t(1-t^{2})}{x(1-x^{2})} \right]^{2} \right)}{1-\left[ \frac{t(1-t^{2})}{x(1-x^{2})} \right]^{2}}\frac{\D t}{1-t^2},
\end{align} where we have a partial fraction\begin{align}\begin{split}
\frac{4}{1-t^2}\frac{\left[ \frac{t(1-t^{2})}{x(1-x^{2})} \right]^{2}}{1-\left[ \frac{t(1-t^{2})}{x(1-x^{2})} \right]^{2}}={}&-\frac{1}{1-3x^{2}}\sum_{\ell,m\in\{0,1\}}\left[(-1)^{\ell}x-\frac{(-1)^{m}i\left( 3x^2 -2\right)}{\sqrt{3x^{2}-4}} \right]\frac{1}{t-\sigma_{\ell,m}(x)}\\{}&-\frac{2x}{1-3x^{2}}\left( \frac{1}{t-x}-\frac{1}{t+x} \right).\label{eq:pf1}
\end{split}
\end{align}Combining this with \eqref{eq:Li2fib6k} in an application of the GPL recursion \eqref{eq:GPL_rec}, we reach the last column of    (A.\AAAref{eq:Li3sum6k5}).

Thanks to the second half of \eqref{eq:MezoG}, we may  evaluate the series in
(A.\AAAref{eq:Li3sum6k6}) by\begin{align}
2\int_0^1\left[\frac{\Li_2\left( \frac{t(1-t^{2})}{x(1-x^{2})} \right)}{1-\frac{t(1-t^{2})}{x(1-x^{2})}} - \frac{\Li_2\left(- \frac{t(1-t^{2})}{x(1-x^{2})} \right)}{1+\frac{t(1-t^{2})}{x(1-x^{2})}}\right] \frac{\frac{t(1-t^{2})}{x(1-x^{2})}\D t}{1-t^2}.
\end{align}    Meanwhile, we note that
\begin{align}
\begin{split}
\frac{2}{1-t^2}\frac{\frac{t(1-t^{2})}{x(1-x^{2})}}{1+(-1)^{\ell}\frac{t(1-t^{2})}{x(1-x^{2})}}={}&\frac{(-1)^{\ell}}{1-3x^{2}}\sum_{m\in\{0,1\}}\left[(-1)^{\ell}x-\frac{(-1)^{m}i\left( 3x^2 -2\right)}{\sqrt{3x^{2}-4}} \right]\frac{1}{t-\sigma_{\ell,m}(x)}\\{}&-\frac{2 x}{1-3 x^2 }\frac{1}{t+(-1)^{\ell}x}
\end{split}
\end{align}and [cf.\ \eqref{eq:LiSx+}--\eqref{eq:LiSx-}]\begin{align}
\Li_2
\left(\pm\frac{t\big(1-t^{2}\big)}{x\big(1-x^{2}\big)} \right)=-\sum _{\alpha\in\{-1,0,1\}}\sum _{w\in S^\pm_{x}}G(\alpha ,w;t)
\end{align}both hold. Therefore, the last column  of   (A.\AAAref{eq:Li3sum6k6}) is a result of the GPL recursion \eqref{eq:GPL_rec}.

We verify (A.\AAAref{eq:Li3sum6k5+6}) in two steps.

First, we sum over \eqref{eq:log_sqr_1-tt} to get\begin{align}\begin{split}
&
\sum_{k=1}^\infty\frac{\binom{2k}{k}\left[ \left(\mathsf H_{k-1}-\mathsf H_{2k-1}+2\log2\right)^2 +\mathsf H_{k-1}^{(2)}-5\mathsf H_{2k-1}^{(2)}+\frac{2\pi^2}{3}\right]}{k\binom{3k}{k}\binom{6k}{3k}}\left[ \frac{2^{2}}{x\big(1-x^{2}\big)} \right]^{2k}
\\={}&4\int_0^1\frac{\left[ \frac{t(1-t^{2})}{x(1-x^{2})} \right]^{2}\log^2\frac{t^{2}}{1-t^{2}}}{1-\left[ \frac{t(1-t^{2})}{x(1-x^{2})} \right]^{2}}\frac{\D t}{1-t^2}\\={}&\sum_{\ell,m\in\{0,1\}}\left[(-1)^{\ell}x-\frac{(-1)^{m}i\left( 3x^2 -2\right)}{\sqrt{3x^{2}-4}} \right]\frac{A_3(\sigma_{\ell,m}(x))}{3x^{2}-1}+\frac{2x[A_{3}(x)-A_{3}(-x)]}{3x^{2}-1}\end{split}\label{eq:log_sqr_int}
\end{align}after falling back on the definition of $A_3(w) $ in \eqref{eq:A3defn} and the  partial fraction expansion in \eqref{eq:pf1}.

Second, after reassembling the two equations in \eqref{eq:MezoG}, we get \begin{align}
\sum_{k=1}^\infty(\mathsf H_k-\mathsf H_{2k})z^{2k}=\frac{1}{2}\left[ \frac{\Li_1(z)}{1+z}+\frac{\Li_1(-z)}{1-z} \right]
\end{align}for $ |z|<1$, so  \eqref{eq:Itab2b} and \eqref{eq:Itab2c} bring us \begin{align}\small
\begin{split}&
\sum_{k=1}^\infty\frac{\left(\mathsf H_{k-1}-\mathsf H_{2k-1}+2\log2\right)(\mathsf H_k-\mathsf H_{2k})\binom{2k}{k}}{k\binom{3k}{k}\binom{6k}{3k}}\left[ \frac{2^{2}}{x\big(1-x^{2}\big)} \right]^{2k}\\={}&-2\int_0^1\left[\frac{\Li_1\left( \frac{t(1-t^{2})}{x(1-x^{2})} \right)}{1+\frac{t(1-t^{2})}{x(1-x^{2})}}+\frac{\Li_1\left(- \frac{t(1-t^{2})}{x(1-x^{2})} \right)}{1-\frac{t(1-t^{2})}{x(1-x^{2})}}\right]\frac{\log\frac{t^{2}}{1-t^{2}}\D t}{1-t^2}\\={}&\sum_{w\in S^+_x}\left\{\sum_{\tau\in\{-1,1\}}\mathscr A_3(w,\tau)\tau -\sum_{m\in\{0,1\}}\left[x-\frac{(-1)^{m}i\left( 3x^2 -2\right)}{\sqrt{3x^{2}-4}} \right]\frac{\mathscr A_3(w,\sigma_{0,m}(x))}{3x^2-1}+\frac{2x\mathscr A_3(w,-x)}{3x^{2}-1}\right\}
\\{}&+\sum_{w\in S^-_x}\left\{\sum_{\tau\in\{-1,1\}}\mathscr A_3(w,\tau)\tau -\sum_{m\in\{0,1\}}\left[-x-\frac{(-1)^{m}i\left( 3x^2 -2\right)}{\sqrt{3x^{2}-4}} \right]\frac{\mathscr A_3(w,\sigma_{1,m}(x))}{3x^2-1}-\frac{2x\mathscr A_3(w,x)}{3x^{2}-1}\right\},\end{split}\label{eq:log_Li1_int}
\end{align}where the bivariate function $ \mathscr A_3(a,b)$ is defined in \eqref{eq:AA3_defn}. Since the closed form evaluation\begin{align}\small
\begin{split}&
\sum_{k=1}^\infty\frac{\left[\left(\mathsf H_{k-1}-\mathsf H_{2k-1}+2\log2\right)\big( \!-\!\frac{1}{2k} +2\log2\big)+\frac{2\pi^2}{3}\right]\binom{2k}{k}}{k\binom{3k}{k}\binom{6k}{3k}}\left[ \frac{2^{2}}{x\big(1-x^{2}\big)} \right]^{2k}\\={}&\sum_{\tau,\tau'\in\{-1,1\}}\sum_{w\in S_{x}}G(\tau',\tau,w;1)\tau-2\sum_{\tau\in\{-1,1\}}\sum_{w\in S_{x}}G(0,\tau,w;1)\tau\\{}&+\left\{\frac{2}{1-3 x^2}\left[ x \mathfrak{l}(x)-\frac{i \left(3 x^2-2\right) \mathfrak{m}(x)}{\sqrt{3 x^2-4}} \right]+\left[3\left(  \tanh ^{-1}\frac{1}{x} \right)^{2}
+\frac{1}{4}\log^{2} \frac{x^2-2+i \sqrt{3 x^2-4}}{x^2-2-i \sqrt{3 x^2-4}}\right]\right\}\log2\\{}&+\frac{2 \pi ^2}{3\left( 1-3 x^2 \right)}\left[ 6 x \tanh ^{-1}\frac{1}{x}+\frac{i \left(3 x^2-2\right)
}{\sqrt{3 x^2-4}} \log \frac{x^2-2+i \sqrt{3 x^2-4}}{x^2-2-i \sqrt{3 x^2-4}}\right]
\end{split}\label{eq:log_misc}
\end{align}is available from entries (A.\AAref{eq:sum6}), (A.\AAref{eq:sum6''}), and (A.\AAref{eq:sum6'''}) of
Table \ref{tab:6kLi2} together with entries (A.\AAAref{eq:Li3sum6k2}) and (A.\AAAref{eq:Li3sum6k3})  of Table \ref{tab:6kLi3}, we may compute Table \ref{tab:6kLi3}(A.\AAAref{eq:Li3sum6k5+6}) via the linear combination \eqref{eq:log_sqr_int}$-$\eqref{eq:log_Li1_int}$-$\eqref{eq:log_misc}.

 \item We may evaluate the infinite series in  (B.\BBBref{eq:Li3sum6k1'}) through the following integral:\begin{align}\frac12
\int_{0}^1\left[\Li_2
\left(\frac{t\big(1-t^{2}\big)}{x\big(1-x^{2}\big)} \right) -\Li_2
\left(-\frac{t\big(1-t^{2}\big)}{x\big(1-x^{2}\big)} \right)\right]\frac{t\big(1-t^{2}\big)}{x\big(1-x^{2}\big)}\frac{\D t}{t^2}.
\end{align}Since we can verify \begin{align}\begin{split}
\int_0^1\frac{1-t^2}{t}G(\alpha ,w;t)\D t={}&G(0,\alpha,w;1)+\frac{\alpha^{2}-1}{2}G(\alpha,w;1)+\frac{(1-w)\alpha}{2}\log\left( 1-\frac{1}{w} \right)\\{}&+\frac{1-w^{2}}{4}\log\left( 1-\frac{1}{w} \right)-\frac{\alpha}{2}-\frac{w}{4}-\frac{1}{8}
\end{split}
\end{align}in Panzer's \texttt{HyperInt} \cite{Panzer2015}, and have the vanishing sums\begin{align}
\sum _{\alpha\in\{-1,0,1\}}\alpha,\quad \sum_{w\in S^\pm_x}\log\left( 1-\frac{1}{w} \right),\quad \sum_{w\in S^\pm_x} w,\label{eq:vanishingS}
\end{align}the last column of  (B.\BBBref{eq:Li3sum6k1'}) emerges immediately.

 In view of \eqref{eq:2k-1_int_repn2} and \eqref{eq:2k-1_int_repn3}, we may construct the following integral representations for  (B.\BBBref{eq:Li3sum6k2'}) and (B.\BBBref{eq:Li3sum6k3'}):\begin{align}
&\int_{0}^1\log\left( \frac{1+\frac{t(1-t^{2})}{x(1-x^{2})}}{1-\frac{t(1-t^{2})}{x(1-x^{2})}} \right)\frac{t\big(1-t^{2}\big)}{x\big(1-x^{2}\big)}\frac{\log t\D t}{t^2},\\&\int_{0}^1\log\left( \frac{1+\frac{t(1-t^{2})}{x(1-x^{2})}}{1-\frac{t(1-t^{2})}{x(1-x^{2})}} \right)\frac{t\big(1-t^{2}\big)}{x\big(1-x^{2}\big)}\frac{\log \left(1-t^{2}\right)\D t}{2t^2}.
\end{align} In parallel to \eqref{eq:loglogA} and \eqref{eq:loglogB}, we have \cite{Panzer2015}{\allowdisplaybreaks
\begin{align}
\begin{split}
&\int_{0}^1\frac{\big(1-t^{2}\big)\log\left( 1-\frac{t}{w} \right)}{t}\log t\D t\\={}&G(w,1,1;1)+G(0,w,1;1)+\frac{w^2G(w,1;1)}{2}-\frac{1}{6}\log^3\left( 1-\frac{1}{w} \right)\\{}&-\frac{w^{2}}{4}\log^2\left( 1-\frac{1}{w} \right)+\left( \frac{1-w^2}{4} -\frac{\pi^2}{6}\right)\log\left( 1-\frac{1}{w} \right)-\frac{3w+1}{4},
\end{split}\\\begin{split}&
\int_{0}^1\frac{\big(1-t^{2}\big)\log\left( 1-\frac{t}{w} \right)}{t}\log \big(1-t^{2}\big)\D t\\={}&G(0,-1,w;1)+G(0,1,w;1)+G(0,w,-1;1)+G(0,w,1;1)\\{}&-\frac{1-w^2}{2}\left[ G(w,-1;1)+G(w,1;1)-\log\left( 1-\frac{1}{w} \right) \right]-\frac{3-2\log2}{2}w-\frac{1}{2},
\end{split}
\end{align}
}hence the GPL forms of   (B.\BBBref{eq:Li3sum6k2'}) and (B.\BBBref{eq:Li3sum6k3'}).

From the generating function in \eqref{eq:Hkloglog}, we know that \begin{align}
\sum_{k=1}^\infty \frac{\mathsf H_{2k-2}}{2k-1}z^{2k-1}=\frac{\log^2(1-z)-\log^2(1+z)}{4}
\end{align}holds for $|z|<1$. Setting $ z=\frac{t(1-t^{2})}{x(1-x^2)}$ in the equation above, we may convert the infinite series for (B.\BBBref{eq:Li3sum6k4'}) into\begin{align}\frac14
\int_0^1\left[ \log^2 \left( 1-\frac{t(1-t^{2})}{x(1-x^{2})} \right)-\log^2 \left( 1+\frac{t(1-t^{2})}{x(1-x^{2})} \right)\right]\frac{t\big(1-t^{2}\big)}{x\big(1-x^{2}\big)}\frac{\log t\D t}{t^2}.
\end{align}Bearing in mind that  \cite{Panzer2015} \begin{align}
\begin{split}&
\int_{0}^1\frac{\big(1-t^{2}\big)\log\left( 1-\frac{t}{a} \right)}{t}\log\left( 1-\frac{t}{b} \right)\D t\\={}&G(0,a,b;1)+G(0,b,a;1)-\frac{1-a^{2}}{2}G(a,b;1)-\frac{1-b^2}{2}G(b,a;1)\\{}&+\frac{1-a^2-2ab+2b}{4}\log\left( 1-\frac{1}{a} \right)+\frac{1-b^2-2ab+2a}{4}\log\left( 1-\frac{1}{b} \right)-\frac{3(a+b)+1}{4}
\end{split}
\end{align}and that sums like \eqref{eq:vanishingS} should vanish, we get the GPL representation of  (B.\BBBref{eq:Li3sum6k4'}).

One can work out the following integral representations for (B.\BBBref{eq:Li3sum6k5'}) and (B.\BBBref{eq:Li3sum6k6'}):\begin{align}
\int_0^1\frac{\left[ \frac{t(1-t^{2})}{x(1-x^{2})} \right]^{2}\Li_2\left(\left[ \frac{t(1-t^{2})}{x(1-x^{2})} \right]^{2} \right)}{1-\left[ \frac{t(1-t^{2})}{x(1-x^{2})} \right]^{2}}\frac{\D t}{t^2},&\\\int_0^1\left[\frac{\Li_2\left( \frac{t(1-t^{2})}{x(1-x^{2})} \right)}{1-\frac{t(1-t^{2})}{x(1-x^{2})}} - \frac{\Li_2\left(- \frac{t(1-t^{2})}{x(1-x^{2})} \right)}{1+\frac{t(1-t^{2})}{x(1-x^{2})}}\right] \frac{\frac{t(1-t^{2})}{x(1-x^{2})}\D t}{2t^2},&
\end{align}and recast them into the tabulated GPL forms after taking partial fractions. We omit the details.

For (B.\BBBref{eq:Li3sum6k5'+6'}), simply turn  [cf.\ \eqref{eq:log_sqr_tt}]
\begin{align}
\begin{split}
&\sum_{k=1}^\infty\frac{\left[ \left(\mathsf H_{k}-\mathsf H_{2k}+\frac{2}{2k-1}+2\log2\right)^2 +\mathsf H_{k}^{(2)}-5\mathsf H_{2k}^{(2)}+\frac{4}{(2k-1)^{2}}+\frac{2\pi^2}{3}\right]\binom{2k}{k}}{(2k-1)\binom{3k}{k}\binom{6k}{3k}}\left[ \frac{2^{2}}{x\big(1-x^{2}\big)} \right]^{2k}
\\={}&\int_0^1\frac{\left[ \frac{t(1-t^{2})}{x(1-x^{2})} \right]^{2}\log^2\frac{t^{2}}{1-t^{2}}}{1-\left[ \frac{t(1-t^{2})}{x(1-x^{2})} \right]^{2}}\frac{\D t}{t^2},
\end{split}\end{align}and\begin{align}\begin{split}
&\sum_{k=1}^\infty\frac{\left(\mathsf H_{k}-\mathsf H_{2k}+\frac{2}{2k-1}+2\log2\right)\left(\mathsf H_{k}-\mathsf H_{2k}
-\frac{1}{2k}\right)\binom{2k}{k}}{(2k-1)\binom{3k}{k}\binom{6k}{3k}}\left[ \frac{2^{2}}{x\big(1-x^{2}\big)} \right]^{2k}\\={}&-\int_0^1\left[\frac{\Li_1\left( \frac{t(1-t^{2})}{x(1-x^{2})} \right)}{1+\frac{t(1-t^{2})}{x(1-x^{2})}}-\frac{\Li_1\left(- \frac{t(1-t^{2})}{x(1-x^{2})} \right)}{1-\frac{t(1-t^{2})}{x(1-x^{2})}}\right]\frac{\frac{t(1-t^{2})}{x(1-x^{2})}\log\frac{t^{2}}{1-t^{2}}\D t}{2t^2},
\end{split}
\end{align}
into GPL forms via the tricks designed for the proof of (A.\AAAref{eq:Li3sum6k5+6}), before evaluating \begin{align}\small
\begin{split}
\sum_{k=1}^\infty\frac{\left[\left(\mathsf H_{k}-\mathsf H_{2k}+\frac{2}{2k-1}+2\log2\right)\left( \frac{1}{2k}+\frac{2}{2k-1}+2\log2\right)+\frac{4}{(2k-1)^{2}}+\frac{2\pi^2}{3}\right]\binom{2k}{k}}{(2k-1)\binom{3k}{k}\binom{6k}{3k}}\left[ \frac{2^{2}}{x\big(1-x^{2}\big)} \right]^{2k}
\end{split}\end{align}through entries (A.\AAref{eq:sum6}), (A.\AAref{eq:sum6''}), (A.\AAref{eq:sum6'''}), (D.\DDref{eq:sum6(2k-1)}), (D.\DDref{eq:sum6(2k-1)''}), and  (D.\DDref{eq:sum6(2k-1)'''}) in Table \ref{tab:6kLi2}, as well as entries (B.\BBBref{eq:Li3sum6k1'})--(B.\BBBref{eq:Li3sum6k3'}) in Table \ref{tab:6kLi3}.
\qedhere\end{enumerate}
\end{proof}
Some explicit computations of cyclotomic multiple zeta values bring us the corollary below.
\begin{corollary}[cf.\ {\cite[Conjecture 5.5(iii)]{Sun2023}}]
The closed-form evaluations in Table \ref{tab:6kSSO} hold true. Consequently, we have \eqref{eq:conj5.5iii}.\end{corollary}\begin{table}[t]\caption{Evaluations of
selected infinite series by $ \mathfrak Z_k(8),k\in\{1,2,3\}$, where $ \theta\colonequals e^{\pi i/4}$, $\lambda\colonequals\log2 $, $\widetilde \lambda\colonequals\log\big(1+\sqrt{2}\big)$, and $ G\colonequals\I\Li_2(i)$\label{tab:6kSSO}}\begin{scriptsize}\newcounter{SSO}
\begin{tabular}{c|L|L}\hline\hline {\No}&a_k&\displaystyle\sum_{k=1}^\infty\f{a_{k}\bi{2k}k}
{\bi{3k}k\bi{6k}{3k}}2^{3k}\vphantom{\frac{\frac\int1}{}}\\[12pt]\hline  {\refstepcounter{SSO}\SSOlabel{eq:6kLi3(8)a}}\text{(\theSSO{})}&\frac{1}{(2k-1)^3}&\sqrt{2}\RE\left[ \Li_{1,1,1}(i,1,\theta)- \Li_{1,1,1}(i,1,-\theta)+\Li_3(\theta)\right]+\frac{3\zeta(3)}{128\sqrt{2}}-\frac{\RE\Li_{1,1}(i,\theta)}{\sqrt{2}}-\frac{\pi\I\Li_2(\theta)}{\sqrt{2}}+\frac{\pi G}{2\sqrt{2}}\vphantom{\frac{\int1}{}}\\&&{}-\frac{6 \lambda ^2 \widetilde{\lambda }-8 \widetilde{\lambda }^3-12 \widetilde{\lambda }^2+12 \lambda  \widetilde{\lambda }-144 \widetilde{\lambda }-9 \lambda ^2}{96 \sqrt{2}}-\frac{3 \pi }{8 \sqrt{2}}-\frac{\pi ^2 \big(12 \lambda-22 \widetilde{\lambda } +13\big)}{384 \sqrt{2}}\\[5pt]  {\refstepcounter{SSO}\SSOlabel{eq:6kLi3(8)b}}\text{(\theSSO{})}&\frac{\mathsf H_{k-1}^{(2)}-2\mathsf H_{2k-1}^{(2)}}{k}&\frac{8\sqrt{2}}{5}\RE\left[3\Li_{1,1,1}(i,1,\theta)- 3\Li_{1,1,1}(i,1,-\theta) +13 \Li_3(\theta)\right]\\&&{}+4\sqrt{2}\I\left[ 8\Li_{1,1,1}(i,1,\theta)-3\Li_{1,1,1}(i\theta,1,i) -11\Li_{1,1,1}(\theta,1,\theta)\right]+\frac{39 \zeta (3)}{80 \sqrt{2}}-\frac{\sqrt{2}\big(110 \widetilde \lambda+25\lambda+12 \pi \big)\I\Li_2(\theta)}{5}  \\&&{}+\frac{360 G \widetilde{\lambda }+60 G \lambda-3 \lambda ^2 \widetilde{\lambda }+28 \widetilde{\lambda }^3 }{10 \sqrt{2}}-\frac{\pi  \big(384 G+160 \lambda  \widetilde{\lambda }-380 \widetilde{\lambda }^2+65 \lambda ^2\big)}{80 \sqrt{2}}-\frac{\pi ^2 \big(13 \widetilde{\lambda }+6 \lambda \big)}{40 \sqrt{2}}-\frac{139 \pi ^3}{240 \sqrt{2}}\\[5pt]  {\refstepcounter{SSO}\SSOlabel{eq:6kLi3(8)c}}\text{(\theSSO{})}
&\frac{\mathsf H_{k-1}^{(2)}-2\mathsf H_{2k-1}^{(2)}}{2k-1}&-\frac{2\sqrt{2}}{5}\RE\left[9\Li_{1,1,1}(i,1,\theta)- 9\Li_{1,1,1}(i,1,-\theta)
+\frac{67}{3}\Li_3(\theta)\right]-\frac{\I\left[ 8\Li_{1,1,1}(i,1,\theta)-3\Li_{1,1,1}(i\theta,1,i) -11\Li_{1,1,1}(\theta,1,\theta)\right]}{\sqrt{2}}\\&&{}-\frac{67 \zeta (3)}{320 \sqrt{2}}+\frac{\big(110 \widetilde \lambda+25\lambda+72 \pi\big)\I\Li_2(\theta) }{20 \sqrt{2}}-\frac{180 G \widetilde{\lambda }+30 G \lambda-9 \lambda ^2 \widetilde{\lambda }+44 \widetilde{\lambda }^3 }{40 \sqrt{2}}+\frac{\pi  \big(384 G+160 \lambda  \widetilde{\lambda }-380 \widetilde{\lambda }^2+65 \lambda ^2\big)}{640 \sqrt{2}}\\&&{}-\frac{\pi ^2 \big(\widetilde{\lambda }-18 \lambda\big)}{160 \sqrt{2}}+\frac{139 \pi ^3}{1920 \sqrt{2}}
\\[5pt]\hline  {\refstepcounter{SSO}\SSOlabel{eq:6kLi3(8)b'}}\text{(\theSSO{})}&\frac{\mathsf H_{k-1}^{(2)}-5\mathsf H_{2k-1}^{(2)}}{k}&12\sqrt{2}\RE\left[\Li_{1,1,1}(i,1,\theta)- \Li_{1,1,1}(i,1,-\theta) \right]+\frac{68\sqrt{2}}{5}\RE\Li_3(\theta)\vphantom{\dfrac{\int1}{}}\\&&{}+\frac{2\sqrt{2}}{5}\I\left[8\Li_{1,1,1}(i,1,\theta) -3\Li_{1,1,1}(i\theta,1,i) -11\Li_{1,1,1}(\theta,1,\theta)\right]+\frac{51 \zeta (3)}{160\sqrt{2}}-\frac{\sqrt{2}\big(22 \widetilde{\lambda }+5 \lambda +60 \pi \big)\I\Li_2(\theta)}{10}\\&&{} +\frac{72 G \widetilde{\lambda }+12 G \lambda-15 \lambda ^2 \widetilde{\lambda }-4 \widetilde{\lambda }^3 }{20 \sqrt{2}}+\frac{\pi  \big(384 G-32 \lambda  \widetilde{\lambda }+76 \widetilde{\lambda }^2-13 \lambda ^2\big)}{160 \sqrt{2}}+\frac{\pi ^2 \big(79 \widetilde{\lambda }-30 \lambda \big)}{80 \sqrt{2}}-\frac{47 \pi ^3}{480 \sqrt{2}}\\[5pt] {\refstepcounter{SSO}\SSOlabel{eq:6kLi3(8)c'}}\text{(\theSSO{})}&\frac{\mathsf H_{k}^{(2)}-5\mathsf H_{2k}^{(2)}}{2k-1}&-9\sqrt{2}\RE\left[\Li_{1,1,1}(i,1,\theta)- \Li_{1,1,1}(i,1,-\theta) \right]-\frac{143\sqrt{2}}{15}\RE\Li_3(\theta)-\frac{\I\left[ 8\Li_{1,1,1}(i,1,\theta)-3\Li_{1,1,1}(i\theta,1,i) -11\Li_{1,1,1}(\theta,1,\theta)\right]}{10\sqrt{2}}\\&&{}-\frac{143 \zeta (3)}{640 \sqrt{2}}+\frac{\big(22 \widetilde{\lambda }+5 \lambda +360 \pi \big)\I\Li_2(\theta)}{40\sqrt{2}}-\frac{36 G \widetilde{\lambda }+6 G \lambda-45 \lambda ^2 \widetilde{\lambda }+28 \widetilde{\lambda }^3+160 \widetilde{\lambda } }{80 \sqrt{2}}-\frac{\pi  \big(4224 G-32 \lambda  \widetilde{\lambda }+76 \widetilde{\lambda }^2-13 \lambda ^2-640\big)}{1280 \sqrt{2}}\\&&{}-\frac{\pi ^2 \big(197 \widetilde{\lambda }-90 \lambda \big)}{320 \sqrt{2}}+\frac{47 \pi ^3}{3840 \sqrt{2}}+\frac{\pi ^2-12 \widetilde{\lambda }^{2}}{8} \\[5pt]\hline\hline\end{tabular}\end{scriptsize}\end{table}\begin{proof}Consider the sets [cf.\ \eqref{eq:Sx_defn}--\eqref{eq:Sx-defn}]\begin{align}
S_{\sqrt{2}}^+=\left\{ \sqrt{2},e^{3\pi i/4},e^{-3\pi i/4} \right\},\quad S_{\sqrt{2}}^-=\left\{- \sqrt{2},e^{\pi i/4},e^{-\pi i/4} \right\},
\end{align}and $ S_{\sqrt{2}}^{}=S^+_{\sqrt{2}}\cup S^-_{\sqrt2}$.  In Au's   \texttt{MultipleZetaValues} package \cite{Au2022a}, the commands \texttt{IterIntDoableQ[\{0, 1, -1, w\}]} and
\texttt{IterIntDoableQ[\{0, 1, -1, a, b\}]} return the natural number  $8$ whenever  $ w\in S_{\sqrt2}^{}$, $ a\in S^+_{\sqrt{2}}$, and $ b\in S_{\sqrt{2}}^-$. This indicates that $G(0,\alpha,w;1)\in\mathfrak Z_3(8) $, $ G(0,w;1)\in\mathfrak Z_2(8)$, $ \log\big( 1-\frac{1}{w} \big)=G(w;1)\in\mathfrak Z_1(8)$, and $G(w,\tau,w';1)\in\mathfrak Z_3(8) $ for all $ w\in
S^\pm_{\sqrt{2}}$, $w'\in S^\mp_{\sqrt{2}}$, $ \alpha\in\{-1,0,1\}$, $ \tau\in\{-1,1\}$.

According to the last paragraph, all the summands in the last column of Table \ref{tab:6kLi3}(B.\BBBref{eq:Li3sum6k1'}) are expressible through members of $ \mathfrak Z_k(8),k\in\{1,2,3\}$  when $ x=\sqrt2$. Therefore, one can check  Table \ref{tab:6kSSO}\eqref{eq:6kLi3(8)a} against \texttt{MZExpand} in Au's   \texttt{MultipleZetaValues} package \cite{Au2022a}.

The linear combination (A.\AAAref{eq:Li3sum6k5})$ -2\times$(A.\AAAref{eq:Li3sum6k6}) from  Table \ref{tab:6kLi3} consists of  terms that are all expressible through members of $\mathfrak Z_3(8)$, according to Au's \texttt{IterIntDoableQ} test. This explains Table \ref{tab:6kSSO}\eqref{eq:6kLi3(8)b}. Likewise, one arrives at  Table \ref{tab:6kSSO}\eqref{eq:6kLi3(8)c} after taking  (B.\BBBref{eq:Li3sum6k5'})$ -2\times$(B.\BBBref{eq:Li3sum6k6'}) from Table \ref{tab:6kLi3}.

The aforementioned \texttt{IterIntDoableQ}  tests also entail [cf.\ \eqref{eq:A3defn}--\eqref{eq:AA3_defn} for the definitions of the functions $A_3(w)$ and $\mathscr A_3(a,b)$] \begin{align}
\begin{cases}A_3(w)\in\mathfrak Z_3(8), & \text{if }w\in S_{\sqrt2}, \\
\mathscr A_3(w,\tau)\in\mathfrak Z_3(8), & \text{if }w\in S_{\sqrt2},\tau\in\{-1,1\},\\ \mathscr A_3(a,b)\in\mathfrak Z_3(8),&\text{if }a\in S^\pm_{\sqrt{2}},b\in S^\mp_{\sqrt{2}}.
\end{cases}
\end{align}Therefore,
each term in Table \ref{tab:6kLi3}(A.\AAAref{eq:Li3sum6k5+6}) is explicitly computable in Au's    \texttt{MultipleZetaValues} package \cite{Au2022a}, which brings us Table \ref{tab:6kSSO}\eqref{eq:6kLi3(8)b'}.
Similarly, one can deduce Table \ref{tab:6kSSO}\eqref{eq:6kLi3(8)c'} from Table \ref{tab:6kLi3}(B.\BBBref{eq:Li3sum6k5'+6'}).

Combining the information above with the closed-form evaluations of [see  entries (A.\AAref{eq:sum6}) and (A.\AAref{eq:sum6'''}) in Table \ref{tab:6kLi2}]\begin{align}
\sum_{k=1}^\infty\frac{\binom{2k}k2^{3k}}{k^{s}\binom{3k}k\binom{6k}{3k}},\quad s\in\{1,2\}
\end{align}
as well as
[see \eqref{eq:T60spec1}
and
\eqref{eq:T60spec4} in
Corollary
\ref{cor:Z2(8)}]\begin{align}\sum_{k=1}^\infty\frac{\binom{2k}k2^{3k}}{(2k-1)^{s}\binom{3k}k\binom{6k}{3k}},\quad s\in\{1,2\},
\end{align} we get \eqref{eq:conj5.5iii}  as claimed.
\end{proof}

\end{document}